\documentclass{amsart}

\usepackage{amsmath}
\usepackage{amsthm}
\usepackage{amssymb}
\usepackage{mathtools}
\mathtoolsset{showonlyrefs}
\usepackage{enumerate}
\usepackage{stmaryrd}
\usepackage{xspace}
\usepackage{macros}
\usepackage{hyperref}
\usepackage{mathrsfs}
\usepackage{tikz}

\begin{document}


\title[Two-scale methods for integrodifferential operators]{Monotone two-scale 
methods for a class of integrodifferential operators and 
applications}

\author[J.P.~Borthagaray]{Juan Pablo Borthagaray}
\address[J.P.~Borthagaray]{Instituto de Matematica y Estad\'istica ``Rafael Laguardia'', Facultad de Ingenier\'ia, Universidad de la Rep\'ublica, Montevideo, Uruguay.}
\email{jpborthagaray@fing.edu.uy}
\urladdr{https://sites.google.com/view/jpborthagaray}
\thanks{JPB has been supported in part by Fondo Clemente Estable grant 2022-172393.}

\author[R.H.~Nochetto]{Ricardo H. Nochetto}
\address[R.H.~Nochetto]{Department of Mathematics and Institute for Physical Science and
Technology, University of Maryland, College Park, MD 20742, USA.}
\email{rhn@math.umd.edu}
\urladdr{https://www.math.umd.edu/~rhn/}
\thanks{RHN has been supported in part by NSF grant DMS-1908267.}

\author[A.J.~Salgado]{Abner J. Salgado}
\address[A.J.~Salgado]{Department of Mathematics, University of Tennessee, Knoxville, TN 37996, USA.}
\email{asalgad1@utk.edu}
\urladdr{https://math.utk.edu/people/abner-salgado/}
\thanks{AJS has been supported in part by NSF grant DMS-2111228.}

\author[C.~Torres]{C\'eline Torres}
\address[C.~Torres]{Department of Mathematics, University of Maryland, College Park, MD 20742, USA.}
\email{cetorres@umd.edu}
\thanks{CT has been supported in part by NSF grant DMS-1908267 and the Swiss 
National Science Foundation grant number P500PT\_203082}

\thanks{Part of this work was completed while RHN and AJS were in residence at the Institute for Computational and Experimental Research in Mathematics (ICERM) in Providence, RI, during the Numerical PDEs: Analysis, Algorithms, and Data Challenges semester program. ICERM is supported by NSF grant DMS-1929284.}




\begin{abstract}
We develop a monotone, two-scale discretization for a class of integrodifferential operators of order $2s$, $s \in (0,1)$. We apply it to develop numerical schemes, and derive pointwise convergence rates, for linear and obstacle problems governed by such operators. As applications of the monotonicity, we provide error estimates for free boundaries and a convergent numerical scheme for a concave fully nonlinear, nonlocal, problem.
\end{abstract}

\maketitle

\section{Introduction}

In recent times, nonlocal models have gained a lot of popularity in the pure and applied sciences. The reason behind this boom is manifold. From the point of view of applications, it is claimed that they are able to encode a wider range of phenomena when compared to their local counterparts. In this regard, for instance, one can refer to the creation of peridynamics \cite{MR1727557,MR2348150}, nonlocal diffusion reaction equations \cite{Yamamoto2012}, fractional Cahn-Hilliard models \cite{Akagi2016, Ainsworth2017a}, fractional porous media equations \cite{Pablo2012}, the fractional Schr\"odinger equation \cite{Treeby2010}, fractional viscoelasticity \cite{Carmichael2015}, fractional Monge-Amp\`ere \cite{Caffarelli2015} and many more. The interested reader can refer to the many existing overviews \cite{Bonito2018,MR4043885,MR4189291} for further references and insight.

In our opinion, nonlocal models started to gain the attention of the mathematical community after the seminal work \cite{Caffarelli2007}, where the authors showed that the so-called fractional Laplacian (in the whole space) can be realized as a degenerate elliptic operator in one more dimension. With this, many of the techniques that were used to analyze local problems could now be applied to nonlocal ones. Later, purely nonlocal techniques for many problems were also introduced; see, for instance, \cite{Caffarelli2017,MR3503212}.

Regarding the numerical treatment of nonlocal problems, some early attempts could be found in \cite{Huang2014,MR2252038,MR2300467,MR2212226,MR2300288}. However, references \cite{Nochetto2015,MR3356020,Bonito2019,Acosta2017a} deserve special attention. In \cite{Nochetto2015} the extension technique of \cite{Caffarelli2007} was exploited to develop a numerical method, and its analysis, for the so-called \emph{spectral} fractional Laplacian. References \cite{MR3356020,Bonito2019} for the spectral and integral fractional Laplacian, respectively, develop a nonconforming approximation on the basis of an integral representation, a quadrature formula, and spatial discretization. Finally, closer to our discussion here is \cite{Acosta2017a} where the \emph{integral fractional Laplacian} is considered. On the basis of weighted H\"older regularity results developed in \cite{RosOton2014}, the authors of this work provide weighted Sobolev regularity, and construct a direct discretization over graded meshes that, in two dimensions, is optimal with respect to degrees of freedom.

The error estimates of \cite{Acosta2017a, Borthagaray2021} are in the energy and $L^2(\Omega)$--norms. Pointwise error estimates for the integral fractional Laplacian were obtained in \cite{Han2021}, where a monotone two-scale method for this operator was developed. The error estimates in \cite{Han2021} hinged on weighted H\"older regularity and the construction of suitable barriers.

Two-scale methods, such as those developed in \cite{Nochetto2018, 
Li2018,Nochetto2019, Nochetto2019b} and \cite{Han2021,Han2023} naturally inherit a 
discrete maximum principle from the continuous operator. This is in general 
not true for finite element approximations. For instance, the finite element 
approximation of the classical Laplacian possesses a discrete maximum principle 
only under certain geometric constraints on the mesh such as weak acuteness \cite[Section 
III.20]{MR1115237}. We stress that for nonlocal operators, the lack of 
monotonicity for the finite element approximation is not expected to be fixed 
by only a geometric constraint as for its local counterpart, especially when $s \approx 0$, where
the stiffness matrix approaches a standard mass matrix.

This brings us then to the main goals of the present work. We will consider an integrodifferential operator of order $2s$, $s \in (0,1)$, of the form
\begin{equation}
\label{eq:TheOperator}
  \calL_\eta[w](x) = \vp \int_{\R^d} \frac{ w(x) - w(y)}{|x-y|^{d+2s}} \ani\left(\frac{x-y}{|x-y|}\right) \diff{y},
\end{equation}
where the so-called \emph{kernel} $\eta :\polS^{d-1} \to \Real$ is assumed to verify several properties that will be specified below. Operators of this form have been extensively studied in probability and finance, as they represent the generator of \(2s\)--stable L\'evy processes, see \eg \cite{Cont2006,Bertoin1996}. The simplest, and most important, example is the aforementioned integral fractional Laplacian $\Laps$, for which $\eta(\theta)$ is constant
\begin{equation}
\label{eq:DefOfLaps}
  \eta(\theta) \equiv C(d,s) = \frac{4^s s\Gamma(s+d/2) }{ \pi^{d/2} \Gamma(1-s) } > 0.
\end{equation}
In \cite{Han2021}, the authors propose a monotone discretization for the fractional Laplacian based on splitting the integral in the definition of the operator into a singular and a tail parts; the former is approximated by a scaled finite-difference Laplacian. This, in essence, amounts to truncating the kernel in a neighborhood of the origin.
Inspired by that work, for a broader class of operators with kernel $\eta$ satisfying $0<\lambda\le \eta(\theta)\le \Lambda$ (\cf Definition \ref{def:ClassCoeff}), we develop a monotone, two-scale method for the discretization of the operator \eqref{eq:TheOperator} and show its consistency under weighted H\"older regularity assumptions. Instead of truncation, we propose to \emph{regularize} the kernel in \eqref{eq:TheOperator} near the origin. From an analysis perspective, this means that we use a zero-order operator for the approximation of the singular part of the integral. Regularization is robust with respect to both the presence of $\eta$ and the limit $s\uparrow1$. In Appendix~\ref{sec:Regularize} we provide further insight on the effect of truncation, which boils down to replacing the singular integral part by a scaled second order differential operator, which turns out to be {\it anisotropic} when $\eta$ is variable and may thus require wide stencils for finite difference approximations to be monotone.
  
We then consider three applications of our scheme of increasing difficulty. First, we obtain pointwise rates of convergence for a linear problem where the operator is of the form \eqref{eq:TheOperator}. Next, we consider an obstacle problem for \eqref{eq:TheOperator} and obtain pointwise rates of convergence. The latter, in turn, allow us to derive error estimates in Hausdorff distance for free boundaries. This endeavor hinges on nondegeneracy properties which we further proved.  Finally, we develop a numerical scheme for a class of concave fully nonlinear integrodifferential operators of order $2s$. Rates of convergence for problems of this type, however, are at this stage beyond our reach.
We point out that our rates of convergence rely on realistic H\"older regularity assumptions: $C^s$ regularity up to the boundary and interior $C^\beta$ regularity, with $\beta \in (2s,4)$. In the linear case, the interior regularity is solely dictated by the regularity of the forcing function, whereas the obstacle problem exhibits a reduced interior regularity of order $\beta \le 1+s$.

Our paper is organized as follows.
We introduce notation and a suitable functional framework in Section~\ref{sec:Prelim}. We also present the class of integrodifferential operators of interest to us along with their most elementary properties. We explain the two-scale discretization of our operators in Section~\ref{sec:TwoScales}. We describe the action of the two scales, namely regularization and discretization, and study the consistency of approximations. The first application of our two-scale discretization is the content of Section~\ref{sec:Linear}, where we approximate a linear problem and derive pointwise rates of convergence. We next study an obstacle problem in Section~\ref{sec:Obstacle}. We first prove pointwise convergence rates in Section \ref{sub:ObstacleDiscrete}. In Section~\ref{sub:FreeBoundaryObstacle} we establish nondegeneracy properties, and in Section \ref{sub:GammaEstimates} we obtain error estimates for free boundaries in Hausdorff distance. These estimates account for the presence of both regular and singular points. We insist once again that all  our estimates rely on realistic regularity assumptions. As a final application of our scheme, Section~\ref{sec:HJB} studies a concave fully nonlinear integrodifferential equation. We propose a convergent scheme, albeit without explicit rates of convergence. We conclude with Appendix~\ref{sec:Regularize} that sheds light on the nature of the differential operator required by consistency with quadratics to replace the tip of the singular kernel $\frac{1}{|z|^{d+2s}}\eta\big(\frac{z}{|z|}\big)$ for a general function $\eta$.

\section{Notation and preliminaries}
\label{sec:Prelim}

Let us introduce some notation and terminology that will be used throughout our discussion. For $A,B \in \R$ the relation $A \lesssim B$ means that, for a nonessential constant $c$, we have $A \leq c B$. The value of this constant may change in every occurrence. $A \gtrsim B$ means $B \lesssim A$. If $A \lesssim B \lesssim A$ we abbreviate this by saying $A \approx B$. The Landau symbols, big-O and little-o, respectively, are $\calO$ and $\frako$.

For $r >0$ and $x \in \Real^d$ we denote by $B_r(x)$ the (open) Euclidean ball of radius $r$ centered at $x$. We set $B_r= B_r(0)$. The unit sphere in $\R^d$ is $\polS^{d-1} = \partial B_1$.

Throughout our discussion $\Omega \subset \R^d$, with $d \in \polN$, is a bounded Lipschitz domain which we assume satisfies an exterior ball condition. We denote its boundary by $\partial\Omega$ and its complement by $\Omega^c = \Real^d \setminus \overline\Omega$.

For $x \in \Omega$ we define \(\delta(x)=\dist(x,\partial\Omega)\) and if $x,y \in \Omega$, then we set \(\delta(x,y) =\min\{\delta(x),\delta(y)\}\). 

\subsection{Function spaces and their norms}
\label{sub:Functions}

We will adhere to standard notation for the Lebesgue spaces, and their norms, when they are defined either over the whole space or some domain. Since we are concerned with pointwise estimates, we must work with functions that are at least continuous. The space of functions $w : \bar\Omega \to \Real$ that are continuous is denoted by $C(\bar\Omega)$. We recall that this space endowed with the norm
\[
  \| w \|_{C(\bar\Omega)} = \sup_{x \in \bar\Omega} |w(x)|,
\]
is a Banach space. We also need spaces of continuously differentiable functions. For $k \in \polN$ we set
\[
  C^k(\bar\Omega)= \left\{ w : \bar\Omega \to \Real \ \middle| \ D^\beta w \in C(\bar\Omega), |\beta| \leq k \right\}.
\]
We also set $C^0(\bar\Omega) = C(\bar\Omega)$. The norm on these spaces is
\[
  \| w \|_{C^k(\bar\Omega)} = \sup_{|\beta|\leq k } \| D^\beta w\|_{C(\bar\Omega)}.
\]
To provide a more refined characterization of smoothness, for $k \in \polN_0$, and $\alpha \in (0,1]$ we define the H\"older spaces via
\[
  C^{k,\alpha}(\bar\Omega) = \left\{ w \in C^k(\bar\Omega) \ \middle| \ [D^\beta w]_{C^{0,\alpha}(\bar\Omega)} < \infty, |\beta| = k \right\},
\]
where
\[
  [w]_{C^{0,\alpha}(\bar\Omega)} = \sup_{x,y \in \bar\Omega: x\neq y} \frac{ |w(x) - w(y)|}{|x-y|^\alpha }.
\]
The norm in the H\"older spaces is
\[
  \| w \|_{C^{k,\alpha}(\bar\Omega)} = \max\left\{ \| w \|_{C^k(\bar\Omega)}, \max_{|\beta|=k} [D^\beta w]_{C^{0,\alpha}(\bar\Omega)} \right\}.
\]
These spaces are complete.
We set $C^{k,0}(\bar\Omega) = C^k(\bar\Omega)$. In addition, whenever $k \in \polN_0$ and $\alpha \in (0,1]$ with $k+\alpha \notin \polN_0$, we may denote $C^{k+\alpha}(\bar\Omega) = C^{k,\alpha}(\bar\Omega)$. Finally, for $k \in \polN_0$ and $\alpha \in (0,1]$ we will say that $w \in C^{k,\alpha}(\Omega)$ if $w_{|\bar U} \in C^{k,\alpha}(\bar U)$ for all $U \Subset \Omega$. A further refinement of these spaces will be detailed when needed.

We will also deal with fractional Sobolev spaces. For $r \in (0,1)$ we set
\[
  H^r(\Real^d) = \left\{ w\in L^2(\Real^d) \ \middle| \ \| w \|_{H^r(\Real^d)}<\infty \right\},
\]
with
\begin{align*}
  \|w \|_{H^r(\Real^d)}^2 &= \| w \|_{L^2(\Real^d)}^2 + [w]_{H^r(\R^d)}^2, \\
  [w]_{H^r(\R^d)}^2 &=
  \frac{C(d,r)}2 \iint_{\Real^d \times \Real^d} \frac{| w(x)-w(y)|^2}{|x-y|^{d+2s}} \diff{x} \diff{y},
\end{align*}
for which it is a Hilbert space.
We also set, again for $r \in (0,1)$,
\[
  \tildeH^r(\Omega) = \left\{ w \in \Ldeux \ \middle| \ \tildew \in H^r(\R^d) \right\},
\]
where by $\tildew$ we denote the extension by zero onto $\Omega^c$. Owing to the fractional Poincar\'e inequality,
\[
  \| w \|_\Ldeux \lesssim [\tildew]_{H^r(\R^d)}, \quad \forall w \in \tildeH^r(\Omega),
\]
the seminorm $[\cdot]_{H^r(\R^d)}$ is actually an equivalent norm in $\tildeH^r(\Omega)$ which makes it Hilbert. The dual of $\tildeH^r(\Omega)$ is denoted by $H^{-r}(\Omega)$, and the duality pairing will be $\langle \cdot, \cdot \rangle$.
In what follows, if confusion does not arise, we shall suppress the explicit mention of zero extensions.

\subsection{The integrodifferential operator}
\label{sub:AssumptionsOnKernel}

We shall consider integrodifferential operators of the form \eqref{eq:TheOperator}. Regarding the kernel, we encode its assumptions in the following definition.

\begin{defn}[class $\calC(\lambda,\Lambda)$]
\label{def:ClassCoeff}
Let $\lambda, \Lambda \in (0,\infty)$ with $\lambda \leq \Lambda$. We will say that the kernel $\ani : \polS^{d-1} \to \Real$ belongs to the class $\calC(\lambda,\Lambda)$ if it is:
\begin{enumerate}[1.]
  \item Symmetric, \ie $\eta(\theta) = \eta(-\theta)$ for all $\theta \in \polS^{d-1}$.

  \item Elliptic, \ie we have
  \[
    0 < \lambda \leq \eta(\theta) \leq \Lambda, \quad \mae \quad \theta \in \polS^{d-1}.
  \]
\end{enumerate}
\end{defn}

Owing to these assumptions, if $\ani \in \calC(\lambda, \Lambda)$, a fractional integration by parts shows that
\[
  [ w ]_{H^s(\Real^d)}^2 \approx \langle \calL_\eta[w], w \rangle, \quad \forall w \in \tildeH^s(\Omega),
\]
with equivalence constants that depend on $\ani$ only through the constants $\lambda$ and $\Lambda$. In other words, for every $\eta \in \calC(\lambda,\Lambda)$ the expression
\[
  \| w \|_{\eta,s}^2 = \iint_{\R^d \times \R^d} \frac{|w(x)- w(y)|^2}{|x-y|^{d+2s}} \eta \left(\frac{x-y}{|x-y|}\right) \diff{y} \diff{x},
\]
is an equivalent Hilbertian norm on $\tildeH^s(\Omega)$.

Owing to the positivity of the kernel, the operator \eqref{eq:TheOperator} satisfies a comparison principle.

\begin{prop}[comparison principle]
\label{prop:Compare}
Let $K \colon \R^d \to \R$ be a positive function. Define the operator $\calL_K$ via
\[
\calL_K [w](x) = \vp \int_{\R^d} (w(x) - w(y)) \, K(x-y) \diff{y} .
\]
Let $w \colon \R^d \to \R$ be such that, in the weak sense, $\calL_K[w] \geq 0$ almost everywhere in $\Omega$ and $w \geq 0$ in $\Omega^c$. Then, $w \geq 0$ in $\Omega$.
\end{prop}
\begin{proof}
  See \cite[Proposition 4.1]{RosOton2016a}.
\end{proof}

As a final piece of preliminary notation we must introduce a set of parameters, relations between them, and a space of functions with a prescribed behavior away from the boundary. We let $\beta \in \Real$, and assume it satisfies
\begin{equation}
\label{eq:Parameters}
 \beta > 2s, \qquad \beta, \beta-2s \notin \polN_0. 
\end{equation}
Given such $\beta$ we define
\begin{equation}
\label{eq:ClassForSol}
  \calS^\beta(\Omega) = \left\{ w \in C^\beta(\Omega)\cap C^{0,s}(\bar\Omega) \ \middle| \ 
  \| w \|_{C^{\beta}(\left\{x\in \Omega \ \middle| \ \delta(x)\geq \rho \right\})} \lesssim \rho^{s-\beta}
  \right\}.
\end{equation}
The usefulness of this function class shall become clear once we perform the error analysis. Indeed, as we shall show (\cf Lemma~\ref{lem:hoelderboundary} below), solutions of problems involving the operator $\calL_\eta$ typically belong to such class.

\section{Two-scale discretization}
\label{sec:TwoScales}

We now begin the discretization of our integrodifferential operator \eqref{eq:TheOperator}. Besides consistency, a fundamental necessity, we wish to preserve its comparison property detailed in Proposition~\ref{prop:Compare}.

These two requirements stand at odds of each other. For instance, when dealing with (local) elliptic second order differential operators and their finite element discretization, very stringent mesh requirements must be imposed to retain a comparison principle, see \cite[Section III.20]{MR1115237} and \cite[Section 3.5]{Neilan2017}. If, on the contrary, we discretize via finite differences, it is known that wide stencils must be employed, see \cite[Section 3.2]{Neilan2017}. To fulfill these two conditions then, we will employ a two-scale discretization.

Two-scale discretizations have become a popular choice to develop schemes that preserve the comparison principle for (local) elliptic second order differential operators. As an example, the reader is referred to \cite{Nochetto2018, 
Li2018,Nochetto2019, Nochetto2019b}. Regarding the two-scale discretization of nonlocal problems we mention \cite{Han2021,Han2023}, where the authors proposed a two-scale discretization for the fractional Laplacian.

\subsection{The regularization scale}
\label{sub:Regularization}
The first step in the approximation of \eqref{eq:TheOperator} is regularization. The idea is that, for some $\vare>0$, we split the integral that defines the operator in two parts: $\Real^d \setminus B_\vare$ and $B_\vare$. The integral in the small ball $B_\vare$ is then regularized by introducing a non singular kernel with suitable approximation properties. 

We begin by observing that, since the kernel $\eta$ is symmetric, we have
\begin{equation}
\label{eq:etaISsymmetric}
  \calL_\eta[w](x) = \frac12 \int_{\Real^d}\left( 2w(x) - w(x-y) - w(x+y) \right) \eta\left( \frac{y}{|y|} \right) \calK(|y|) \diff{y},
\end{equation}
where $\calK(r) = \tfrac1{r^{d+2s}}$. Let now $\vare>0$. We define
\begin{equation}
\label{eq:TheOperatorvare}
  \calL_{\eta,\vare}[w](x) = \frac12 \int_{\Real^d} \left( 2w(x) - w(x-y) - w(x+y) \right) \eta\left( \frac{y}{|y|} \right) \calK_\vare(|y|) \diff{y},
\end{equation}
where the radial, nonsingular kernel $\calK_\vare$ is defined as
\begin{equation}
\label{eq:DefOfKvare}
  \calK_\vare(r) = \begin{dcases}
                 \frac1{r^{d+2s}}, & r \geq \vare, \\
                 \frac1{\vare^{d+2s}} + \gamma (r^2 - \vare^2) + \nu (r^3 -\vare^3), & r<\vare.
               \end{dcases}
\end{equation}
We point out that \eqref{eq:DefOfKvare} gives
\[
  \lim_{r \uparrow \vare} \calK_\vare(\vare) = \vare^{-d-2s}, \qquad \lim_{t \downarrow 0}  \frac{\calK_\vare(t) - \calK_\vare(0)}{t}= 0.
\]
The constants $\gamma$ and $\nu$ are chosen so that $\calK_\vare \in C^1([0,\infty))$ and
\[
  \calL_{\eta}[q] = \calL_{\eta,\vare}[q], \quad \forall q \in \polP_2.
\]
The smoothness requirement implies that
\[
  2\gamma \vare + 3 \nu \vare^2 = -(d+2s)\vare^{-d-2s-1}.
\]
On the other hand, we will have exactness provided that
\[
  \int_0^\vare \frac{r^2}{r^{d+2s}}r^{d-1}\diff{r} = \int_0^\vare r^2 \calK_\vare(r)r^{d-1} \diff{r}.
\]
In short, the parameters $\gamma$ and $\nu$ have the values
\begin{equation}
\label{eq:DefOfRegParameters}
  \begin{aligned}
    \gamma  &= -\frac{(4 + d) (d + 2 s) (3 + d + 2 s) }{ 4 (1 - s)}\varepsilon^{- d - 2 s - 2} < 0,\\
    \nu &= \frac{(5 + d) (d + 2 s) (2 + d + 2 s)}{ 6 (1 - s)} \varepsilon^{-d - 2 s - 3} > 0.
  \end{aligned}
\end{equation}

The following result explores the interior consistency of this regularization. By interior we mean that we consider points $x \in \Omega$, such that $B_\vare(x) \subset \Omega$, \ie all the points where the regularization of the kernel takes place are contained in $\Omega$.

\begin{thm}[interior consistency]
\label{thm:ConsistencyRegularization}
Let $\Omega$ be a bounded Lipschitz domain that satisfies the exterior ball 
condition. Assume that $\beta \in (2s, 4]$ and $\eta 
\in \calC(\lambda,\Lambda)$. Let $w \in \calS^\beta(\Omega)$, where this class is 
defined in \eqref{eq:ClassForSol}. Let $\alpha_0>1$, $\vare>0$, and $x \in \Omega$ be such 
that  $\delta(x) \geq \alpha_0 \vare$. Then, it holds that
\[
  \big| \calL_{\eta}[w](x) - \calL_{\eta,\vare}[w](x) \big| \lesssim \vare^{\beta - 2s} \delta(x)^{s-\beta}.
\]
The implicit constant depends on $d$, $s$, $\alpha_0$, $\beta$, $\lambda$, and $\Lambda$.
\end{thm}
\begin{proof}
We need to estimate
\begin{align*}
  \big| \calL_{\eta}[w](x) &- \calL_{\eta,\vare}[w](x) \big| \\
 & \hskip-0.3cm  =  \frac12 \left| \int_{B_\vare} \left( 2w(x) - w(x-y) - w(x+y) \right) \eta\left( \frac{y}{|y|} \right) \left( \frac1{|y|^{d+2s}} - \calK_\vare(|y|) \right) \diff{y}\right|  \\
& \hskip-0.3cm \lesssim \int_0^\vare \| w \|_{C^\beta(B_r(x))} r^\beta \left| \frac1{r^{d+2s}} - \calK_\vare(r) \right| r^{d-1} \diff{r} \\
 & \hskip-0.3cm\lesssim \int_0^\vare \left( \delta(x) - r\right)^{s-\beta} \left| \frac1{r^{d+2s}} - \calK_\vare(r) \right| r^{\beta+d-1} \diff{r},
\end{align*}
where, in the last step, we used the interior H\"older estimate in \eqref{eq:ClassForSol} satisfied by $w\in\mathcal{S}^\beta(\Omega)$.
Since $\delta(x) \geq \alpha_0 \vare$, we get
\begin{multline*}
  \int_0^\vare \left( \delta(x) - r\right)^{s-\beta} \left| \frac1{r^{d+2s}} - \calK_\vare(r) \right| r^{\beta+d-1} \diff{r}  \\ \lesssim \delta(x)^{s-\beta} \int_0^\vare r^{\beta+d-1} \left( \frac1{r^{d+2s}} + \calK_\vare(r) \right) \diff{r}
  \lesssim \vare^{\beta - 2s} \delta(x)^{s-\beta},
\end{multline*}
which is the desired estimate.
\end{proof}

As it can be seen from the proof of Theorem~\ref{thm:ConsistencyRegularization}, the interior consistency of our regularization depends on how close we are to the boundary. In particular, we need to have $B_\vare(x) \Subset\Omega$. For this reason, given $\vare>0$ we let $\bvare : \Omega \to (0, \infty)$ be sufficiently smooth and satisfy
\[
  \bvare(x) \leq \min\left\{ \frac{\delta(x)}2, \vare \right\}, \quad \forall x \in \Omega.
\]
With this function at hand, we then define
\[
  \calL_{\eta,\bvare}[w](x) = \calL_{\eta,\bvare(x)}[w](x).
\]
In order to properly leverage the boundary regularity of solutions, the choice of $\bvare$ is made more precise below; depending on the problem under consideration.

Owing to the positivity of $\calK_\vare$, the comparison principle from Proposition~\ref{prop:Compare} holds for $\calL_{\eta,\bvare}$. Thus, pointwise consistency estimates for this operator can be achieved by combining a comparison principle with a suitable barrier function. We thus construct a barrier.

\begin{lemma}[barrier]
\label{lem:barrier}
Let $\Omega$ be a bounded, Lipschitz domain. Define
\[
  b(x) = \chi_\Omega(x), \quad \forall x \in \Omega.
\]
Then,
\[
  \delta(x)^{-2s} \lesssim \calL_{\eta,\bvare}[b](x), \quad \forall x \in \Omega.
\]
\end{lemma}
\begin{proof}
Since, by construction, we have $\bvare(x) \leq \delta(x)$ for all $x \in \Omega$ we may write
\[
  \calL_{\eta,\bvare}[b](x) = \int_{\Real^d} \left( 1 - b(y) \right) \eta\left( \frac{x-y}{|x-y|} \right) \calK_{\bvare(x)}(|x-y|) \diff{y} \geq  \lambda \int_{\Omega^c} \frac1{|x-y|^{d+2s}} \diff{y}.
\]
Since
\[
  \delta(x)^{-2s} \lesssim \int_{\Omega^c} \frac1{|x-y|^{d+2s}} \diff{y},
\]
with a hidden constant that depends on $d,$ $s,$ and $\Omega$, the result follows.
\end{proof}

We are now in position to estimate the consistency of our regularized operator.

\begin{thm}[consistency]
\label{thm:Consistency}
In the setting of Theorem~\ref{thm:ConsistencyRegularization} assume, in addition,  that $w,w_\vare \in \calS^\beta(\Omega)$ verify
\[
  \calL_\eta[w] = \calL_{\eta,\bvare}[w_\vare] \quad \text{ in } \Omega, \qquad w = w_\vare = 0, \quad \text{ in } \Omega^c.
\]
Then we have
\[
  \| w - w_\vare \|_\Linf \lesssim \max\{ \vare^s, \vare^{\beta - 2s} \}.
\]
\end{thm}
\begin{proof}
Let $x \in \Omega$. Since $\bvare(x) \leq \tfrac12 \delta(x)$, arguing as in the proof of Theorem~\ref{thm:ConsistencyRegularization} we get
\begin{align*}
  \calL_{\eta,\bvare}[w-w_\vare](x) &= \calL_{\eta,\bvare}[w](x) - \calL_\eta[w](x) \lesssim \| w \|_{C^\beta(B_{\bvare(x)}(x))} \bvare(x)^{\beta-2s} \\
  &\lesssim (\delta(x) - \bvare(x))^{s-\beta} \bvare(x)^{\beta-2s} \lesssim 2^{s-\beta} \delta(x)^{s-\beta} \bvare(x)^{\beta-2s} \\
  &\lesssim \delta(x)^{3s-\beta} \bvare(x)^{\beta-2s} \delta(x)^{-2s}.
\end{align*}
If $3s-\beta \geq 0$ continue the estimate as
\[
  \calL_{\eta,\bvare}[w - w_\vare](x) \lesssim \vare^{\beta-2s} \delta(x)^{-2s} \lesssim \vare^{\beta-2s} \calL_{\eta,\bvare}[b](x),
\]
where we used the barrier function of Lemma~\ref{lem:barrier}. If, on the other hand, $3s-\beta < 0$ we use that
\[
  \delta(x)^{3s-\beta}\bvare(x)^{\beta-2s} = \left( \frac{\bvare(x)}{\delta(x)} \right)^{\beta - 3s} \bvare(x)^s \lesssim \vare^s.
\]
Gathering both cases we conclude that, for every $x \in \Omega$, we have
\[
  \calL_{\eta,\bvare}[w-w_\vare](x) \lesssim \max\{\vare^{\beta-2s},\vare^s \} \calL_{\eta,\bvare}[b](x).
\]
Applying Proposition \ref{prop:Compare} (comparison principle) to $\calL_{\eta,\bvare}$ gives the assertion.
\end{proof}

\subsection{The discretization scale}
\label{sub:FEM}

We assume that $\Omega$ is a convex polytope and let $\Triang = \{T \}$ be a conforming and shape regular simplicial triangulation of $\Omega$. The elements $T \in \Triang$ are assumed to be closed. We set, for $T \in \Triang$, $h_T = \diam(T)$. We denote by $\Vert$ the set of vertices of $\Triang$. The interior and boundary vertices are, respectively,
\[
  \VertInt = \Vert \cap \Omega, \qquad \VertBd = \Vert \cap \partial\Omega.
\]
For each interior vertex $z \in \VertInt$ we define its patch to be
\[
  \omega_z = \bigcup\left\{ T \in \Triang \ \middle| \ z \in T \right\}.
\]
By $h_z$ we denote the radius of the ball of maximal radius, centered at $z$, that can be inscribed in $\omega_z$. Over such a triangulation we define the following spaces of functions
\begin{align}
\label{eq:DefOfFeSpace}
  \Fespace &= \left\{ w_\Triang \in C(\bar\Omega) \ \middle| \ w_{\Triang|T} \in \polP_1, \ \forall T \in \Triang \right\}, \\
  \Fespace^0 &= \left\{ w_\Triang \in \Fespace \ \middle| \ w_{\Triang|\partial\Omega} = 0 \right\},
\end{align}
Notice that any function $w_\Triang \in \Fespace^0$ can be trivially extended to $\Omega^c$ by zero. When this causes no confusion, we shall not make a distinction between a function and its extension.

It is a general fact that solutions to problems involving integrodifferential operators, like \eqref{eq:TheOperator}, exhibit an algebraically singular behavior near the boundary, independently of the smoothness of the problem data. The problems that we shall be interested in are no exception; see the regularity results of Sections~\ref{sub:RegLinear} and \ref{sec:Obstacle}. To compensate this we will consider a mesh that is graded towards the boundary as it was first studied in \cite{Acosta2017a} and later in \cite{Bonito2018,Borthagaray2021,Borthagaray2019,Han2021,Han2023}. We consider a mesh size $h>0$ and parameter $\mu \geq 1$. Our mesh $\Triang$ is assumed to satisfy
\begin{equation}
\label{eq:grading}
  h_T \approx \begin{dcases}
                h^\mu, & T \cap \partial\Omega \neq \emptyset, \\
                h\dist(T,\partial\Omega)^{\tfrac{\mu-1}\mu}& T \cap \partial\Omega = \emptyset.
              \end{dcases}
\end{equation}
As shown in \cite[Remark 4.14]{Borthagaray2019} we have that
\begin{equation}
\label{eq:dimFespace}
  \dim \Fespace^0 \approx \begin{dcases}
                            h^{(1-d)\mu}, & \mu \geq \frac{d}{d-1}, \\
                            h^{-d}|\log h|, & \mu = \frac{d}{d-1}, \\
                            h^{-d}, & \mu < \frac{d}{d-1}.
                          \end{dcases}
\end{equation}
Finally, we observe that, under the condition \eqref{eq:grading}, we have 
\begin{equation}\label{eq:NodalPatchSize}
h_z \approx h \delta(z)^{1-1/\mu} \quad \forall z \in \VertInt.
\end{equation}

At this point we impose that, at least, $\frac12 h_z \le \bvare(z)$. We shall later refine this choice; see Definition~\ref{defn:ChoiceBVareLinear} and \ref{defn:ChoiceBVareObstacle} for the linear and obstacle problems, respectively.

\subsubsection{Consistency of interpolation}
\label{subsub:Interpolation}

Let $I_h : C(\bar\Omega) \to \Fespace$ be the Lagrange interpolation operator. For $z \in \VertInt$ we wish to estimate the consistency error
\[
  \calE[w,z] = \calL_{\eta,\bvare}[I_hw](z) - \calL_{\eta}[w](z),
\]
provided the function $w$ possesses suitable, but realistic, smoothness. We begin by rewriting this error as
\[
  \calE[w,z] = \left( \calL_{\eta,\bvare}[w](z) - \calL_\eta[w](z) \right) +  \calL_{\eta,\bvare}[I_h w -w ](z).
\]
The first term entails the regularization error, and it was estimated in Theorem~\ref{thm:ConsistencyRegularization}. Our immediate goal shall be to estimate the second term.

\begin{lemma}[refined interpolation estimate]
\label{lem:IntCbeta}
Let $\beta>0$ satisfy \eqref{eq:Parameters}, $\bar{\beta} = \min\{\beta, 2\}$, and $w \in \calS^\beta(\Omega)$. Furthermore, assume the mesh $\Triang$ satisfies \eqref{eq:grading}. Then, for all $T \in \Triang$, we have
\[
  \| w - I_h w \|_{L^\infty(T)} \lesssim \begin{dcases}
                                           h^{\bar\beta} \dist(T,\Omega)^{s-\bar\beta/\mu}, & T \cap \partial\Omega = \emptyset, \\
                                           h^{\mu s}, &T \cap \partial\Omega \neq \emptyset.
                                         \end{dcases}
\]
Consequently, we have the global pointwise interpolation estimate
\begin{equation}\label{eq:GlobalInterpolation}
  \| w - I_h w \|_{L^\infty(\Omega)} \lesssim \max\{ h^{\mu s} , h^{\bar\beta} \} .
  \end{equation}
\end{lemma}
\begin{proof}
If $w \in C^\alpha(T)$ for some $\alpha \in (0,2]$, we have
\[
  \| w - I_h w \|_{L^\infty(T)} \lesssim \| w \|_{C^{\alpha}(T)} h_T^{\alpha}.
\]
Now, if $T \cap \partial\Omega = \emptyset$ the definition of the class $\calS^\beta(\Omega)$, given in \eqref{eq:ClassForSol}, and the mesh grading imply that
\[
  \| w - I_h w \|_{L^\infty(T)} \lesssim \dist(T,\partial\Omega)^{s-\bar\beta}h^{\bar\beta} \dist(T,\partial\Omega)^{\bar\beta(\mu-1)/\mu} \lesssim h^{\bar\beta} \dist(T,\Omega)^{s-\bar\beta/\mu}.
\]
If, on the contrary, $T \cap \partial\Omega \neq \emptyset$, we estimate
\[
  \| w - I_h w \|_{L^\infty(T)} \lesssim \| w \|_{C^{0,s}(T)} h_T^s \lesssim h^{\mu s}.
\]

Estimate \eqref{eq:GlobalInterpolation} follows from a closer inspection of the case $T \cap \partial\Omega = \emptyset$. If $s \ge \bar\beta/\mu$ then there is nothing to prove, so we assume $s < \bar\beta/\mu$. Let $z \in \VertInt$ be a vertex of $T$. By shape regularity and \eqref{eq:NodalPatchSize}, we have $h_T \approx h_z \approx h \delta(z)^{1-1/\mu},$ and $\dist(T,\partial\Omega) \approx \delta(z)$. Therefore,
we can write
\[
  \| w - I_h w \|_{L^\infty(T)} \lesssim h^{\mu s} h^{\bar\beta - \mu s} \delta(z)^{s-\bar\beta/\mu} \approx h^{\mu s} \left( \frac{h_z}{\delta(z)}\right)^{\bar\beta- \mu s} \lesssim h^{\mu s}. \qedhere
\]
\end{proof}

We can now estimate the remaining consistency term.

\begin{prop}[consistency of interpolation]
\label{prop:Interpolate}
Let the function $\bvare : \Omega \to \Real$ be such that $\tfrac12 h_z \le \bvare(z) \le \tfrac12 \delta(z)$ for all $z \in \VertInt$. In the setting of Lemma~\ref{lem:IntCbeta} we have, for all $z \in \VertInt$,
\begin{equation} \label{eq:InterpolationConsistency}
  \left|\calL_{\eta,\bvare}[I_h w - w](z) \right| \lesssim h^{\bar\beta} \delta(z)^{s-\bar\beta/\mu}\bvare(z)^{-2s} +
  \max\{ h^{\mu s}, h^{\bar\beta} \} \delta(z)^{-2s}.
\end{equation}
\end{prop}
\begin{proof}
Let $z \in \VertInt$. We observe that $I_h w (z) = w(z)$ and that $I_h w \equiv w \equiv 0$ on $\Omega^c$.
Therefore, we have
\[
\calL_{\eta,\bvare}[I_h w - w](z) = \int_\Omega  \left( w(y) - I_h w (y)) \right) \calK_{\bvare(z)}(|z-y|) \eta\left( \frac{z-y}{|z-y|}\right ) \diff{y}  
\]
We define 
\begin{equation}\label{eq:OmegaBdry}
  \Omega^\partial =
    \bigcup \left\{ T \in \Triang \ \middle| \ T \cap \partial\Omega \neq \emptyset \right\},
\end{equation}
and partition the integration domain into
\[
  \bar\Omega = \bigcup_{i=1}^3 D_i,
\]
where
\begin{align*}
  D_1 &= \Omega^\partial  \cap 
  B_{\delta(z)/2}(z)^c, \\
  D_2 &= (\Omega \setminus \Omega^\partial ) \cap 
  B_{\delta(z)/2}(z)^c, \\
  D_3 &= B_{\delta(z)/2}(z),
\end{align*}
and estimate each term separately.

\noindent \underline{Estimate on $D_1$:} Since every point $y \in D_1$ belongs to a boundary-touching element, we use the first part of Lemma~\ref{lem:IntCbeta} to write
\begin{align*}
  & \left| \int_{D_1} \left( w(y) - I_h w (y)) \right) \calK_{\bvare(z)}(|z-y|) \eta\left( \frac{z-y}{|z-y|}\right ) \diff{y} \right|
  \\ & \hskip0.7cm\lesssim h^{\mu s} \int_{D_1} \frac1{|z-y|^{d+2s}} \diff{y} \lesssim \,
  h^{\mu s} \int_{B_{\delta(z)/2}(z)^c} \frac1{|z-y|^{d+2s}} \diff{y} \lesssim h^{\mu s}\delta(z)^{-2s}.
\end{align*}

\noindent \underline{Estimate on $D_2$:} Notice now that every element $y \in D_2$ belongs to a non-boundary-touching element. We can then invoke the other case in the first part of Lemma~\ref{lem:IntCbeta} to estimate
\begin{multline*}
  \left| \int_{D_2} \left( w(y) - I_h w (y)) \right) \calK_{\bvare(z)}(|z-y|) \eta\left( \frac{z-y}{|z-y|}\right ) \diff{y} \right| \lesssim h^{\bar\beta} \int_{D_2}  \frac{ \delta(y)^{s-\bar\beta/\mu} }{|z-y|^{d+2s}} \diff{y}.
\end{multline*}
Now, if $s - \tfrac{\bar\beta}\mu \geq 0$ we simply estimate
\[
  h^{\bar\beta} \int_{D_2}  \frac{ \delta(y)^{s-\bar\beta/\mu} }{|z-y|^{d+2s}} \diff{y} \leq h^{\bar\beta} \int_{B_{\delta(z)/2}(z)^c} \frac1{|z-y|^{d+2s}} \diff{y} \lesssim h^{\bar\beta} \delta(z)^{-2s}.
\]
If, instead, $s - \tfrac{\bar\beta}\mu < 0$ then
\[
  h^{\bar\beta} \int_{D_2}  \frac{ \delta(y)^{s-\bar\beta/\mu} }{|z-y|^{d+2s}} \diff{y} \lesssim
  h^{\bar\beta}h^{\mu (s-\bar\beta/\mu)} \int_{D_2}  \frac{ 1}{|z-y|^{d+2s}} \diff{y} \lesssim
 h^{\mu s} \delta(z)^{-2s}.
\]

\noindent \underline{Estimate on $D_3$:} Observe that $\delta(y) \geq \delta(z) - |z-y|$ and, if $y \in D_3$, we also have that $\tfrac12 \delta(z) \leq \delta(y) \leq \tfrac32 \delta(z)$. Now, since $z \in \VertInt$ we have that $\delta(z) \gtrsim h^\mu$. Consequently, if $y \in T \cap D_3$ for some $T$ such that $T \cap \partial \Omega = \emptyset$,
\[
  h_T \lesssim h \delta(y)^{1-1/\mu} \lesssim h \delta(z)^{1-1/\mu}.
\]
We, once again, invoke the estimates of Lemma~\ref{lem:IntCbeta} to obtain
\begin{align*}
  &\left| \int_{D_3 \cap (\Omega\setminus \Omega^\partial)} \big( w(y) - I_h w (y) \big) \calK_{\bvare(z)}(|z-y|) \eta\left( \frac{z-y}{|z-y|}\right ) \diff{y} \right|  \\
  & \hskip2.5cm \lesssim h^{\bar\beta} \delta(z)^{s-\bar\beta/\mu} \int_{D_3 \cap (\Omega\setminus \Omega^\partial)}  \calK_{\bvare(z)}(|z-y|) \diff{y}  \\
  & \hskip2.5cm \lesssim h^{\bar\beta} \delta(z)^{s-\bar\beta/\mu} \left( \int_{\bvare(z)}^{\delta(z)/2} \frac1{r^{1+2s}} \diff{r} + \int_{B_{\bvare(z)}(z)} \calK_{\bvare(z)}(|z-y|) \diff{y} \right)  \\
  & \hskip2.5cm \lesssim h^{\bar\beta} \delta(z)^{s-\bar\beta/\mu} \bvare(z)^{-2s}.
\end{align*}
On the other hand, if $D_3 \cap \Omega^\partial \neq \emptyset$ it means that $\dist(B_{\delta(z)/2}(z), \partial\Omega) \lesssim h^\mu$, and therefore $\delta(z) = \dist(z, \partial\Omega) \lesssim \delta(z)/2 + h^\mu$, namely  $\delta(z) \approx h^\mu \approx h_z$. Therefore, it must be
 $\bvare(z) \approx \delta(z)$.
In such a case, on $D_3 \cap \Omega^\partial$, we have
\begin{align*}
  & \left| \int_{D_3 \cap \Omega^\partial} \big( w(y) - I_h w (y) \big) \calK_{\bvare(z)}(|z-y|) \eta\left( \frac{z-y}{|z-y|}\right ) \diff{y} \right| \\
  & \hskip2.5cm \lesssim h^{\mu s} \int_{D_3 \cap \Omega^\partial} \calK_{\bvare(z)}(|z-y|) \diff{y} \\
  & \hskip2.5cm \lesssim h^{\mu s} \left( \int_{\bvare(z)}^{\delta(z)/2} \frac1{r^{1+2s}} \diff{r} + \int_{B_{\bvare(z)}(z)} \calK_{\bvare(z)}(|z-y|) \diff{y} \right)  \\
  & \hskip2.5cm \lesssim h^{\mu s} \bvare(z)^{-2s} \approx h^{\mu s} \delta(z)^{-2s}.
\end{align*}
Gathering all the preceding estimates leads to the assertion \eqref{eq:InterpolationConsistency}.
\end{proof}

\section{The linear problem}
\label{sec:Linear}

Having studied a consistent and monotone discretization of the operator $\calL_\eta$, defined in \eqref{eq:TheOperator}, we proceed to use this discretization to propose and analyze numerical methods for problems of increasing complexity. The first one shall be a linear one.

We consider the following problem: Let $s \in (0,1)$ and $\eta \in \calC(\lambda,\Lambda)$ for some $\lambda,\Lambda >0$. Given $f \in \Hsd$, find $u \in \tHs$ such that
\begin{equation}
\label{eq:TheLinearProlem}
  \calL_\eta[u] = f, \qquad \text{ in } \Omega.
\end{equation}
Notice that, by virtue of the definition of the solution space, we are implicitly providing the exterior condition $u=0$ on $\Omega^c$. Owing to the fact that $\| \cdot \|_{\eta,s}$ is an equivalent norm on $\tHs$, existence and uniqueness of a weak solution follows immediately by the Lax-Milgram theorem. In addition, since $\eta$ is positive, a nonlocal maximum principle holds from Proposition~\ref{prop:Compare} (comparison principle).

\subsection{Regularity}
\label{sub:RegLinear}

The regularity properties of $u$, solution of \eqref{eq:TheLinearProlem}, are of utmost relevance for its numerical approximation. In contrast to local elliptic operators, it is well known that solutions to \eqref{eq:TheLinearProlem} possess limited regularity near the boundary, regardless of the smoothness of $\Omega$ and $f$; see \cite{MR0185273,MR3276603,Borthagaray2023}. The following is an \emph{optimal} regularity result.

\begin{lemma}[optimal H\"older regularity]
\label{lem:hoelderest}
Let \(s\in (0,1)\) and let \(\Omega\subset \R^d\) be a bounded Lipschitz
domain that satisfies the exterior ball condition. Let \(f\in L^\infty(\Omega)\) and \(u \in \tHs \) be the (weak) solution to \eqref{eq:TheLinearProlem}. Then we have \(u\in C^{0,s}(\bar{\Omega})\) and
\[
  \|u\|_{C^{0,s}(\bar{\Omega})} \leq C \|f\|_{L^\infty(\Omega)},
\]
where the implicit constant depends only on \(d, \Omega\) and \(s\).
\end{lemma}
\begin{proof}
See \cite[Proposition 4.5]{RosOton2016c}.
\end{proof}

The previous regularity is optimal, as exemplified by existing explicit solutions to the fractional Laplace problem over a ball (\cf \cite{MR0137148}). The reason for this limited regularity gain lies in the fact that there is an algebraic boundary singularity present in the solution, which can be characterized via
\[
  u(x) \approx \dist(x, \partial\Omega)^s,
\]
as $x$ approaches $\partial\Omega$. Higher order regularity estimates can be obtained if one takes into account such boundary behavior. In \cite{RosOton2016} such results were obtained for the fractional Laplacian over weighted H\"older spaces, where the weights are given by powers of the distance to the boundary.

\begin{defn}[weighted H\"older space]
Let \(\sigma> -1\) and \(\beta = k+\gamma >0\) with
\(k\in\N_0, \gamma \in (0,1]\). For \(w\in C^{k}(\Omega)\) define the seminorm
\[
  |w|_{\beta,\Omega}^{(\sigma)} = \sup_{x,y\in\Omega:x \neq y}
  \left(\delta(x,y)^{\beta+\sigma}
  \frac{\left|D^k w(x) - D^k w(y)\right|}{|x-y|^{\gamma}}\right),
\]
and the norm
\[
  \|w\|_{\beta,\Omega}^{(\sigma)} = \sum_{\ell =1 }^{k}
  \sup_{x\in\Omega}\left( \delta(x)^{\ell+\sigma}|D^\ell w(x)|\right) +
  |w|_{\beta,\Omega}^{(\sigma)} +
  \begin{dcases}
      \sup_{x\in\Omega} \left(\delta(x)^\sigma|v(x)|\right), & \sigma\geq 0,\\
      \|v\|_{C^{0,-\sigma}(\bar{\Omega})}, & \sigma<0.
  \end{dcases}
\]
\end{defn}

The methods used in \cite{RosOton2014} can be adapted to operators of the form
\eqref{eq:TheOperator} to obtain the following weighted H\"older regularity.

\begin{lemma}[weighted H\"older estimates]
\label{lem:weightedhoelderest}
Let \(\Omega\) be a bounded domain that satisfies the exterior ball condition, and let $\beta>0$ satisfy \eqref{eq:Parameters}. Let \(f\in C^{\beta-2s}(\Omega) \)
be such that \(\|f\|_{\beta-2s,\Omega}^{(s)}<\infty\). If, for $k \in\polN$, \(\beta\geq k\), then
we additionally assume that \(\eta\in C^k(\polS^{d-1})\). In this setting, the solution of \eqref{eq:TheLinearProlem} satisfies \(u\in C^{\beta}(\Omega)\) and
\[
  \|u\|_{\beta,\Omega}^{(-s)} \lesssim  \|u\|_{C^{0,s}(\R^d)} + \|f\|_{\beta-2s,\Omega}^{(s)} ,
\]
where the implicit constant only depends  on  \(d\), \(\Omega\), \(\lambda\), \(\Lambda\) and \(s\).
\end{lemma}
\begin{proof}
The result can be obtained by using the same procedure as in \cite[Proposition 1.4]{RosOton2014} for $\eta \equiv 1$. For completeness, we repeat the main parts of the argument. 

Let \(x_0\in \Omega\) and \(R= \frac{\delta(x_0)}{K}\) for some \(K\in \N\). We define \(\tilde{u}(x) 
= u(x_0 + Rx)- u(x_0)\) and, following \cite[Proposition 1.4]{RosOton2014}, it is possible to show that
\begin{align}
\label{eq:part1}
  \|\tilde{u}\|_{C^{0,s}(B_1)} &\leq CR^s [u]_{C^{0,s}(B_R(x_0))}\\
  \|(1+|\cdot|)^{-d-2s}\tilde{u}(\cdot)\|_{L^1(\R^n)}  &\leq C R^s [u]_{C^{0,s}(\R^d)}. \label{eq:part2}
\end{align}
In addition, if \(\beta \leq 1\),
\begin{align*}
  \sup_{x,y\in B_1} \frac{|\mathcal{L}_\eta[\tilde{u}](x) - \mathcal{L}_\eta[\tilde{u}](y)|}{|x-y|^{\beta}} 
  &= R^{2s+\beta}\sup_{x,y\in B_R(x_0))} \frac{\left| \mathcal{L}_\eta[u](x) - \mathcal{L}_\eta[u](y) \right|}{|x-y|^{\beta}}.
\end{align*}
We also observe that, since the coefficient $\eta$ is independent of $x$, for any $k \in \polN_0$
\begin{align*}
  D^k\mathcal{L}_\eta[\tilde{u}](x) = R^k \mathcal{L}_\eta[D^k u](x).
\end{align*}
We use this in the case $\beta > 1$ to assert that, for any $\beta>0$ that satisfies \eqref{eq:Parameters},
\begin{align}\label{eq:part3}
  \|\mathcal{L}_\eta[\tilde{u}]\|_{C^\beta(B_1)} \lesssim R^{2s+\beta} \|\mathcal{L}_\eta[u]\|_{C^\beta(B_R(x_0))}
  \leq R^{s} \|\mathcal{L}_\eta[u]\|_{\beta, \Omega}^{(s)}.
\end{align}

The estimates so far do not use the regularity of the coefficient \(\eta\), but only its boundedness and translation invariance. Now we can repeat the proof \cite[Corollary~2.4]{RosOton2014} using the additional regularity on \(\eta\) together with the 
estimates \eqref{eq:part1}---\eqref{eq:part3} to get
\begin{equation}
\label{eq:est1}
  \begin{aligned}
    \|\tilde{u}\|_{C^{\beta}(B_{1/2})} &\lesssim  \|(1+|\cdot|)^{-d-2s}\tilde{u}(\cdot)\|_{L^1(\R^d)} 
      + \|\tilde{u}\|_{C^{\beta-2s}(B_2)}+ \|\mathcal{L}_\eta[\tilde{u}]\|_{C^{\beta-2s}(B_2)} \\
    &\lesssim  R^s\left([u]_{C^{0,s}(\R^d)} + \|f\|_{\beta-2s, \Omega}^{(s)} \right) + \|\tilde{u}\|_{C^{\beta-2s}(B_2)},
  \end{aligned}
\end{equation}
where we assumed that \(K\geq 2\) to bound the term 
\(\|\mathcal{L}_\eta[\tilde{u}]\|_{C^{\beta-2s}(B_2)}\) appropriately. 
We reiterate that the regularity of the coefficient \(\eta\) is only used in order to be able to apply similar arguments to those in the proof of \cite[Corollary~2.4]{RosOton2014}.

Now, for \(\beta-2s \leq s\), the claim follows from estimate \eqref{eq:est1}. Indeed, using that, if \(y\in B_R(x_0)
\), we have
\[
  \|D^{\ell}\tilde{u}\|_{L^\infty(B_{1/2})} = R^{\ell}\|D^{\ell}u\|_{L^\infty(B_{R/2}(x_0))},
\]
and 
\(\delta(x,y)/|x-y| \leq K\) if \(x\) and \(y\) are far away from each other. 

If, instead, we have that \(\beta-2s > s\), we repeat the argument in \eqref{eq:est1} for 
\(\|\tilde{u}\|_{C^{\beta-2s}(B_2)}\) in total $m$ times until \(\beta-2sm\leq s\). Choosing \(K\) large enough, but finite (depending on \(s\)), finishes the proof.
\end{proof}

As a consequence of Lemmas \ref{lem:hoelderest} and \ref{lem:weightedhoelderest}, we have the following regularity estimate away from the boundary.

\begin{lemma}[interior H\"older estimate]
\label{lem:hoelderboundary}
Let \(\Omega\) be a bounded Lipschitz domain that satisfies the exterior ball condition and let \(f\), \(\beta\), and $\eta$ satisfy the same assumptions as in Lemma~\ref{lem:weightedhoelderest}. For every $\rho>0$ we have
\[
  \|u\|_{C^{\beta}(\left\{x\in \Omega : \delta(x)\geq \rho \right\})} \lesssim \rho^{s-\beta},
\]
for a hidden constant only depending on \(\|f\|_{\beta-2s,\Omega}^{(s)},
\|f\|_{L^{\infty}(\Omega)}\), \(s\), and \(\Omega\).
\end{lemma}
\begin{proof}
Repeat the proof of \cite[Corollary 2.5]{Han2021}.
\end{proof}

In short, the previous results show that $u$, the solution of \eqref{eq:TheLinearProlem}, satisfies
\[
u \in \calS^\beta(\Omega), \qquad \|u\|_{\beta,\Omega}^{(-s)} < \infty.
\]

\subsection{Pointwise error estimates}
\label{sub:ErrorLinear}

We have reached the point where we can propose a numerical scheme for \eqref{eq:TheLinearProlem} and provide an error analysis for it. To begin with, we need to make a precise choice of the regularization scale $\bvare$.

\begin{defn}[$\bvare$ for the linear problem]
\label{defn:ChoiceBVareLinear}
For any interior vertex $z \in \VertInt$, we set
\[
  \bvare(z) := \frac12 h_z^{1/2} \delta(z)^{1/2}.
\]
\end{defn}

At this point we make some comments regarding the choice in Definition~\ref{defn:ChoiceBVareLinear}. First, notice that close to the boundary we have
\[
  \bvare(z) \approx h_z \approx \delta(z).
\]
Moreover, $\bvare(z)\le \frac12 \delta(z)$  because $h_z \le \delta(z)$ according to the definition of $h_z$,
whence $B_{\bvare(z)}(z) \subset \Omega$. On the other hand, away from the boundary we have $\delta(z) \approx 1$, so that
\[
  \bvare(z) \approx h_z^{1/2}.
\]
This scaling bears resemblance to the scalings used in two-scale methods for (local) linear second order elliptic problems; see; \cite{Neilan2017,NEILAN2020,MR3939307,Li2018,Nochetto2018,Nochetto2019a,Nochetto2019b}.

We now describe the scheme. For $\bvare$ given according to Definition~\ref{defn:ChoiceBVareLinear}, our scheme seeks for $u_h \in \Fespace^0$ such that
\begin{equation}
\label{eq:TheLinearProlemh}
  \calL_{\eta,\bvare}[u_h](z) = f(z), \quad \forall z \in \VertInt.
\end{equation}
We comment that, since this is a finite dimensional problem, existence and uniqueness of a solution are implied immediately by the following discrete comparison principle for the operator $\calL_{\eta,\bvare}$. 

\begin{prop}[discrete comparison principle]
\label{prop:DiscreteComparison}
Let $\bvare$ be such that, for every $z \in \VertInt$, we have $\bvare(z) \geq \tfrac12 h_z$. Assume that $v_h, w_h \in \Fespace^0$ are such that
\begin{equation} \label{eq:DiscreteComparison}
\calL_{\eta,\bvare}[v_h](z) \ge \calL_{\eta,\bvare}[w_h](z), \quad \forall z \in \VertInt.
\end{equation}
Then,
\begin{equation}\label{eq:DiscreteComparison-outcome}
  v_h(z) \ge w_h(z) \quad \forall z \in \VertInt.
\end{equation}
\end{prop}
\begin{proof}  The proof follows by a simple argument. Suppose the inequality \eqref{eq:DiscreteComparison} is strict and that \eqref{eq:DiscreteComparison-outcome} is false, namely that the function $v_h - w_h$ attains a (non-positive) minimum at an interior node $z \in \VertInt.$ Then, we have
\[
  v_h(z) - w_h(z) \le  v_h(y)  - w_h(y)
\ \implies \
  v_h(z) - v_h(y) \le w_h(z)  - w_h(y) \quad \forall y \in \Real^d
\]
whence we get the contradiction
\[
  \calL_{\eta,\bvare}[v_h](z) \le \calL_{\eta,\bvare}[w_h](z).
\]
This shows that, in case of strict inequality in \eqref{eq:DiscreteComparison}, it must be that $v_h > w_h$ in $\Omega$.

Assume next that the inequality  \eqref{eq:DiscreteComparison} is not strict. Consider the discrete barrier function $b_h = I_h b$, where $b = \chi_\Omega$ and $I_h$ denotes the Lagrange interpolant. We have $\calL_{\eta,\bvare}[b_h](z) > 0$ for all $z \in \VertInt$. 
Therefore, for $\epsilon > 0$, we have the strict inequality
\[
\calL_{\eta,\bvare}[v_h + \epsilon b_h](z) > \calL_{\eta,\bvare}[w_h](z)  \quad \forall z \in \VertInt,
\]
from which it follows that $v_h (z)+ \epsilon b_h (z) > w_h(z)$ for all $z \in \VertInt$. Letting $\epsilon \to 0$, we obtain the desired result.
\end{proof}

We now proceed with the error analysis of our scheme. Using the choice of $\bvare$ given by Definition \ref{defn:ChoiceBVareLinear}, we begin by making the results from Proposition~\ref{prop:Interpolate} more precise.

\begin{prop}[consistency of interpolation for the linear problem]
\label{prop:InterpolateLinear}
Let the regularization scale $\bvare$ verify Definition~\ref{defn:ChoiceBVareLinear}, $\beta$ satisfy \eqref{eq:Parameters}, $\bar{\beta} = \min\{\beta, 2\}$, and $w \in \calS^\beta(\Omega)$. We have
\[
  \big|\calL_{\eta,\bvare}[I_h w - w](z) \big| \lesssim 
  \max\{ h^{\mu s}, h^{\bar\beta - s} \} \delta(z)^{-2s}  \quad \forall z \in \VertInt.
\]
\end{prop}
\begin{proof}
Let $z \in \VertInt$. By \eqref{eq:InterpolationConsistency}, it suffices to show that 
\[h^{\bar\beta} \delta(z)^{s-\bar\beta/\mu}\bvare(z)^{-2s} \lesssim   \max\{ h^{\mu s}, h^{\bar\beta - s} \} \delta(z)^{-2s}. \]
We consider the set $\Omega^\partial$ given by \eqref{eq:OmegaBdry}.

If $z \in \Omega^\partial$, then $\delta(z) \approx h_z \approx h^\mu$. Combining this with $\bvare(z) = \frac12 h_z^{1/2} \delta(z)^{1/2} \approx \delta(z)$, we obtain
\[
h^{\bar\beta} \delta(z)^{s-\bar\beta/\mu}\bvare(z)^{-2s} \approx h^{\mu s} \delta(z)^{-2s}.
\] 
In contrast, if $z \notin \Omega^\partial$, then $z \in T$ for some $T$ with $T \cap \partial \Omega = \emptyset$ and therefore $h_z \approx h_T \approx h \delta(z)^{1-1/\mu}$ and $\bvare(z) = \frac12 h_z^{1/2} \delta(z)^{1/2}$ If $s \ge (\bar\beta-s)/\mu$, then we have
\[
h^{\bar\beta} \delta(z)^{s-\bar\beta/\mu}\bvare(z)^{-2s} \approx h^{\bar\beta -s} \delta(z)^{s-(\bar\beta-s)/\mu}  \delta(z)^{-2s} \lesssim h^{\bar\beta -s} \delta(z)^{-2s}.
\]
Otherwise, we write $h  \approx h_z \delta(z)^{-1+1/\mu}$ and 
\[
h^{\bar\beta} \delta(z)^{s-\bar\beta/\mu}\bvare(z)^{-2s} \approx h_z^{\bar\beta - s} \delta(z)^{-\bar\beta+2s} \delta(z)^{-2s} \approx h^{\mu s}  \left( \frac{h_z}{\delta(z)} \right)^{\bar\beta-s-\mu s} \delta(z)^{-2s}.
\]
Because $h_z \le \delta(z)$ and $\bar\beta-s-\mu s > 0$, the second term in the right hand side is uniformly bounded above.
\end{proof}

It remains then to obtain error estimates. This is the content of the following result.

\begin{thm}[error estimate]
\label{thm:ErrorLinear}
Let $\Omega$ be a convex polytope, $s \in (0,1)$, and $\beta \leq 4$ is such that \eqref{eq:Parameters} holds. Define $\bar\beta = \min\{ \beta, 2\}$. Let $f \in C^{\beta-2s}(\Omega)\cap \Linf$ be such that $\|f\|_{\beta-2s,\Omega}^{(s)}<\infty$. Assume that $\eta \in \calC(\lambda,\Lambda)$ and that if, for some $k \in \polN$, $\beta >k$, then $\eta \in C^{k}(\polS^{d-1})$. Let $u \in \calS^\beta(\Omega)$ solve \eqref{eq:TheLinearProlem} and $u_h \in \Fespace^0$ solve \eqref{eq:TheLinearProlemh}. If $\Triang$ satisfies \eqref{eq:grading} and $\bvare$ is chosen as in Definition~\ref{defn:ChoiceBVareLinear}, then we have
\[
  \| u- u_h \|_\Linf \lesssim \max \big\{ h^{\mu s}, h^{\bar\beta - s}, h^{\beta/2 - s} \big\}.
\]
\end{thm}
\begin{proof}
We split $u - u_h = \left( u - I_h u \right) + \left( I_h u - u_h \right)$. Hence, to prove the claim, by \eqref{eq:GlobalInterpolation} it suffices to estimate the second term. We do so by estimating the consistency
\[
  \left| \calL_{\eta,\bvare}[I_h u - u_h](z) \right| \leq \left| \calL_{\eta,\bvare}[I_h u - u](z) \right| + \left| \calL_{\eta,\bvare}[u - u_h](z) \right| = \mathrm{I} + \mathrm{II},
\]
for $z \in \VertInt$, and then applying Proposition~\ref{prop:DiscreteComparison} (discrete comparison principle).

Proposition \ref{prop:InterpolateLinear} (consistency of interpolation for the linear problem) yields 
\[
\mathrm{I} \lesssim   \max\{ h^{\mu s}, h^{\bar\beta - s} \} \delta(z)^{-2s}  \approx   \max\{ h^{\mu s}, h^{\bar\beta - s} \} \calL_{\eta,\bvare}[b](z),
\]
where we recall that $b$ is the function introduced in Lemma~\ref{lem:barrier} (barrier). Additionally, the choice of $\bvare$, identity \eqref{eq:NodalPatchSize}, and Theorem~\ref{thm:ConsistencyRegularization} (interior consistency) with $\alpha_0 = 2$, give
\begin{align*}
\mathrm{II} = \left| \calL_{\eta,\bvare}[u](z) - \calL_{\eta}[u](z) \right| &\lesssim \bvare(z)^{\beta - 2s} \delta(z)^{s-\beta} \\ & \approx h^{\beta/2 - s} \delta(z)^{s -\frac{\beta/2 -s}\mu} \calL_{\eta,\bvare}[b](z).
\end{align*}
Thus, if $s -\frac{\beta/2 -s}\mu \ge 0$ we are done. In case $s < \frac{\beta/2 -s}\mu$, we write instead $h \approx h_z \delta(z)^{1/\mu - 1}$ and obtain
\begin{align*}
\mathrm{II} &\lesssim h^{\beta/2 - s} \delta(z)^{s -\frac{\beta/2 -s}\mu} \calL_{\eta,\bvare}[b](z) \\ & \approx h^{\mu s}  h_z^{\beta/2 - \mu s- s} \delta(z)^{(1/\mu - 1)(\beta/2 - \mu s- s)} \delta(z)^{s -\frac{\beta/2 -s}\mu} \calL_{\eta,\bvare}[b](z) \\ 
& \approx h^{\mu s}  \left( \frac{h_z}{\delta(z)} \right)^{\beta/2 - \mu s- s} \calL_{\eta,\bvare}[b](z).
\end{align*}
This concludes the proof.
\end{proof}

\begin{remark}[complexity estimate]
\label{rem:ComplexityLinear}
Let us try to interpret the estimates of Theorem~\ref{thm:ErrorLinear} in terms of degrees of freedom. To shorten the notation we set, only for this discussion, $N = \dim \Fespace^0$. Furthermore, we will assume that the right hand side $f$ is as smooth as possible, \ie we let $\beta > 4$ and satisfy \eqref{eq:Parameters}. This regularity assumption implies that $\beta/2 > \bar\beta = 2$.

First, if $d =2$, we can choose $\mu = 2$ so that, according to \eqref{eq:dimFespace}, $N \approx h^{-2}|\log h|$. As a consequence,
\begin{equation}\label{eq:error-N}
  \| u - u_h \|_\Linf \lesssim \max\left\{ N^{-s}|\log N|^{s} , N^{-1+s/2}|\log N|^{1-s/2} \right\}.
\end{equation}
Therefore, with respect to the number of degrees of freedom, and up to logarithmic factors, we obtain convergence with order $s$ for $s \le 2/3$ and with order $1-s/2$ for $s > 2/3$. We point out that for $s\le 2/3$ the error is dominated by the boundary regularity $C^s$ which entails an interpolation error $h^{\mu s} \approx N^{-s}$; hence \eqref{eq:error-N} is optimal. Additionally we observe that, for $s > 2/3$, one does not need to take $\mu = 2$ in two dimensions, and that the maximal convergence rate is attained whenever $\mu s = 2 -s$ \ie $\mu = \frac{2-s}s<2$ for which $N\approx h^{-2}$ according to \eqref{eq:dimFespace}; any extra mesh refinement does not translate in an improvement of convergence rates, although it affects the conditioning of the resulting system.

On the other hand, for $d=3$ we must choose $\mu= \tfrac32$ if we wish to maintain a near optimal number of degrees of freedom, \ie $N \approx h^{-3}|\log h|$. Thus,
\[
  \| u - u_h \|_\Linf \lesssim \max\left\{ N^{-s/2}|\log N|^{s/2} , N^{-2/3+s/3}|\log N|^{2/3-s/3} \right\}.
\]
Again, we observe that, if $-s/2 \le -2/3 + s/3$, namely if $s \ge 4/5$, the maximal convergence order $(2-s)/3$ is attained for meshes graded with $\mu = (2-s)/s<3/2$.
\end{remark}

\begin{remark}[relationship between $\bvare$ and $f$]
\label{rem:FnonSmooth}
Definition \ref{defn:ChoiceBVareLinear} is suitable for sufficiently smooth right hand sides, namely, it formally delivers optimal convergence rates in case $f \in C^{\beta-2s}(\Omega)\cap \Linf$ be such that $\|f\|_{\beta-2s,\Omega}^{(s)}<\infty$ with $\beta \ge 4$. We recall that, since we are using a second difference formula, the interior consistency of the regularized operator cannot exploit any regularity beyond $\calS^4(\Omega)$, cf. Theorem \ref{thm:ConsistencyRegularization}.
Let us briefly comment on what one can obtain when $f$ satisfies the assumptions above but with $\beta \in (2s,4)$ and satisfying \eqref{eq:Parameters}.

We let $\bvare(z) = \frac12 h_z^\alpha \delta(z)^{1-\alpha}$ with $\alpha \in [0,1]$, that clearly satisfies $\frac12 h_z \le \bvare(z) \le \frac12 \delta(z)$. By doing the same calculations as in Proposition \ref{prop:InterpolateLinear} (consistency of interpolation for the linear problem), we obtain
\[
  \left|\calL_{\eta,\bvare}[I_h w - w](z) \right| \lesssim 
  \max\{ h^{\mu s}, h^{\bar\beta - 2 \alpha s} \} \delta(z)^{-2s}  \quad \forall z \in \VertInt.
\]
Arguing then as in the proof of Theorem \ref{thm:ErrorLinear} (error estimate), we obtain
\[
  \| u- u_h \|_\Linf \lesssim \max \{ h^{\mu s}, h^{\bar\beta - 2\alpha s}, h^{\alpha(\beta - 2s)} \}.
\]
Now, if $\beta \in (2,4)$, we have $\bar\beta = 2$ and
\[
2 - 2\alpha s = \alpha(\beta - 2s) \ \Rightarrow \alpha = \frac2\beta.
\] 
Therefore, choosing $\bvare(z) = \frac12 h_z^{\frac2\beta} \delta(z)^{1-\frac2\beta}$ yields the error estimate
\[
  \| u- u_h \|_\Linf \lesssim \max \{ h^{\mu s}, h^{2-\frac{4s}\beta} \}.
\]
In contrast, if $\beta \in (2s, 2)$, we observe $\bar\beta=\beta$ and we get that
\[
  \beta - 2\alpha s = \alpha(\beta - 2s) \quad  \implies \quad \alpha = 1.
\] 
In this low regularity case, setting $\bvare(z) = \frac12 h_z$ gives rise to
\[
  \| u- u_h \|_\Linf \lesssim \max \{ h^{\mu s}, h^{\beta-2s} \}.
\]
The latter will be of interest in the approximation of the obstacle problem in the next section, and justifies Definition \ref{defn:ChoiceBVareObstacle} below.
\end{remark}

\section{The obstacle problem}
\label{sec:Obstacle}

As the next application of our two-scale discretization, we will consider the 
following nonlinear problem. In the setting of Section~\ref{sec:Linear} we 
assume that, in addition, we have $\psi : \bar\Omega \to \Real$ that satisfies 
$\psi <0$ on $\partial\Omega$. We seek for  $u \in \tHs$ that satisfies
\begin{equation}
\label{eq:TheObstacleHJ}
  \min\left\{ \calL_\eta[u] - f, u - \psi \right\} = 0, \quad \mae \ \Omega.
\end{equation}

While existence and uniqueness of a weak solution is classical, the regularity of such solution is more delicate. Following \cite{Borthagaray2019} we introduce the 
classes
\begin{equation}
\label{eq:ObstacleData}
  \calF_s(\bar\Omega) = C^{3-2s+\epsilon}(\bar\Omega), \qquad \Psi = \left\{ \psi \in C(\bar\Omega): \psi_{|\partial\Omega}<0 \right\} \cap C^{2,1}(\Omega),
\end{equation}
where $\epsilon>0$ is sufficiently small, so that $1-2s+ \epsilon \notin \polN$.

For future use we define the contact and non-contact sets as follows:
\[
  \Omega^0 = \left\{ x \in \Omega \ \middle| \ u(x) = \psi(x) \right\}, \qquad
  \Omega^+ = \left\{ x \in \Omega \ \middle| \ u(x) > \psi(x) \right\}.
\]

In order to obtain rates of convergence, we must understand the regularity of the solution. This can be achieved by combining the arguments in \cite{Borthagaray2019}, \cite{Caffarelli2017}, and the regularity of the linear problem presented in Section~\ref{sub:RegLinear}. Namely, one first proves the result for a problem in the whole space $\R^d$, and then use a localization argument; see \cite{Borthagaray2019} for details in the case $\eta \equiv 1$.

\begin{prop}[H\"older regularity]
\label{prop:RegObstacle}
Assume that $f \in \calF_s$ and $\psi \in \Psi$. Then, the solution $u \in \tHs$ of \eqref{eq:TheObstacleHJ} satisfies $u \in C^{1,s}(\Omega)$ and
\[
  \| \calL_\eta[u] \|_{1-s,\Omega}^{(s)} < \infty.
\]
\end{prop}

\begin{remark}[pointwise evaluation]
\label{rem:PtwiseEval}
Notice that, since $\psi<0$ on $\partial\Omega$, the solution to \eqref{eq:TheObstacleHJ} solves the linear problem $\calL_\eta[u] = f$ in a neighborhood of the boundary. By the interior regularity of the previous result we additionally have $\calL_\eta[u] \in C^{0,1-s}(\Omega)$. As a consequence, pointwise evaluation in $\Omega$ of the operator is meaningful.
\end{remark}

\subsection{Pointwise error estimates}
\label{sub:ObstacleDiscrete}

Let us now present a numerical scheme for the obstacle problem \eqref{eq:TheObstacleHJ} and derive pointwise error estimates for it. We seek for $u_h \in \Fespace^0$ such that
\begin{equation}
\label{eq:TheObstacleh}
  \min\left\{ \calL_{\eta,\bvare}[u_h](z) - f(z), u_h(z) - \psi(z) \right\} = 0, \quad \forall z \in \VertInt.
\end{equation}

For this problem, we shall make a different choice of $\bvare$ than for the linear problem. The reason behind this is that even if the data is sufficiently smooth, the solution to the obstacle problem possesses a limited interior regularity, compare Lemma~\ref{lem:weightedhoelderest} (weighted H\"older estimates) with Proposition~\ref{prop:RegObstacle} (H\"older regularity); see also the discussion of the case $\beta \in (2s,2)$ in Remark~\ref{rem:FnonSmooth}.

\begin{defn}[choice of $\bvare$ for the obstacle problem]
\label{defn:ChoiceBVareObstacle}
Let $z \in \VertInt$. We set
\[
  \bvare(z) = \frac12 h_z.
\]
\end{defn}

Existence and uniqueness of $u_h$ follow from the fact that we are in finite dimensions and the comparison principle for $\calL_{\eta,\bvare}$. Of interest here is the derivation of pointwise error estimates. The technique that we will use is rather classical and can be traced back to \cite{MR0488847,MR0488848}, see also \cite{MR3393323}. We begin by introducing the notions of sub- and supersolutions to the obstacle problem.

\begin{defn}[sub- and supersolution]
We say that $u_h^+ \in \Fespace$ is a supersolution to \eqref{eq:TheObstacleh} if $u_h^+ \geq 0$ in $\Omega^c$ and, for all $z \in \VertInt$, we have
\[
  u_h^+(z) \geq \psi(z), \qquad \calL_{\eta,\bvare}[u_h^+](z) \geq f(z).
\]
On the other hand, we say that $u_h^- \in \Fespace$ is a subsolution to \eqref{eq:TheObstacleh} if $u_h^- \leq 0$ in $\Omega^c$ and, for every $z \in \VertInt$, if $u_h^-(z) \geq \psi(z)$, then
\[
  \calL_{\eta,\bvare}[u_h^-](z) \leq f(z).
\]
\end{defn}

Proposition \ref{prop:DiscreteComparison} (discrete comparison principle) for the operator $\calL_{\eta,\bvare}$ gives a comparison for sub- and supersolutions.

\begin{lemma}[discrete comparison]
\label{lem:DiscCompareObstacle} 
Let $u_h^+,u_h^- \in \Fespace$ be super- and subsolutions to \eqref{eq:TheObstacleh}, and $u_h \in \Fespace^0$ be the solution to \eqref{eq:TheObstacleh}. Then, for every $z \in \Vert$ we have
\[
  u_h^-(z) \leq u_h(z) \leq u_h^+(z).
\]
\end{lemma}
\begin{proof}
We consider each inequality separately. Let $u_h^-$ be a subsolution and consider the set of nodes
\[
  C_- = \left\{ z \in \Vert \ \middle| \ u_h^-(z) \geq \psi(z) \right\}.
\]
Now, if $z \in C_-$, we have
\[
  \calL_{\eta,\bvare}[u_h^-](z) \leq f(z) \leq \calL_{\eta,\bvare}[u_h](z).
\]
If, on the other hand $z \notin C_-$, then
\[
  u_h^-(z) < \psi(z) \leq u_h(z).
\]
In summary, the function $w_h = u_h^- - u_h \in \Fespace$ verifies
\[
  \calL_{\eta,\bvare}[w_h](z) \leq 0, \quad z \in C_-, \qquad w_h(z) \leq 0, \quad z \notin C_-.
\]
A variant of Proposition \ref{prop:DiscreteComparison} (discrete comparison principle) then yields $w_h \leq 0$.

Let now $u_h^+$ be a supersolution. Consider now the discrete contact set
\[
  C_+ = \left\{ z \in \Vert \ \middle| \ u_h(z) = \psi(z) \right\},
\]
and observe that, if $z \in C_+$,
\[
  u_h(z) = \psi(z) \leq u_h^+(z).
\]
On the other hand, if $z \notin C_+$ we have
\[
  \calL_{\eta,\bvare}[u_h^+](z) \geq f(z) = \calL_{\eta,\bvare}[u_h](z).
\]
In conclusion, the function $w_h = u_h^+ - u_h \in \Fespace$ satisfies
\[
  \calL_{\eta,\bvare}[w_h] \geq 0, \quad z \notin C_+, \qquad w_h(z) \geq 0, \quad z \in C_+.
\]
Once again, a variant of Proposition \ref{prop:DiscreteComparison} (discrete comparison principle) yields $w_h \geq 0$.
\end{proof}

Next we need to present a suitable discrete proxy for $u$. Namely, we consider $R_hu \in \Fespace^0$ to be such that, for every $z \in \VertInt$,
\begin{equation}
\label{eq:GalerkinProjection}
  \calL_{\eta,\bvare}[R_hu](z) = \calL_\eta[u](z).
\end{equation}
Recall that, as detailed in Remark~\ref{rem:PtwiseEval} (pointwise evaluation), the right hand side is meaningful. The approximation power of $R_hu$ is the content of the following result.

\begin{col}[projection error]
Let $u$ be the solution of \eqref{eq:TheObstacleHJ} and $R_hu$ be defined in \eqref{eq:GalerkinProjection}. If the regularization scale $\bvare$ is chosen according to Definition~\ref{defn:ChoiceBVareObstacle}, then we have
\begin{equation}
\label{eq:ProjError}
  \| u - R_h u \|_\Linf \lesssim \frake(h) = \max\{ h^{\mu s}, h^{1-s} \},
\end{equation}
with an implicit constant that is independent of $h$.
\end{col}
\begin{proof}
We begin by recalling that, as indicated by Proposition~\ref{prop:RegObstacle} (H\"older regularity), we must set $\beta = 1+s$.

Observe that, owing to \eqref{eq:GalerkinProjection},
\[
  \mathrm{I} = \left| \calL_{\eta, \bvare}[u - R_h u] \right| = \left| \calL_{\eta, \bvare}[u] - \calL_\eta[u] \right| \lesssim \bvare(z)^{1-s} \delta(z)^{-1},
\]
where, in the last step, we used Theorem~\ref{thm:ConsistencyRegularization} (interior consistency).

Using the the current choice of $\bvare$ and \eqref{eq:NodalPatchSize} we then continue this estimate as
\begin{align*}
  \mathrm{I} &\lesssim \left( h \delta(z)^{1-\tfrac1\mu} \right)^{1-s} \delta(z)^{-1} = h^{1-s} \delta(z)^{-2s} \delta(z)^{s +\tfrac{s-1}\mu} \lesssim h^{1-s} \delta(z)^{-2s},
\end{align*}
provided that $s + \tfrac{s-1}\mu \geq 0$.

If, on the other hand, we have $s + \tfrac{s-1}\mu < 0$ we use that $\delta(z) \gtrsim h_z \gtrsim h^\mu$ to obtain
\[
  h^{1-s} \delta(z)^{s +\tfrac{s-1}\mu} \lesssim h^{1-s} h^{\mu\left(s +\tfrac{s-1}\mu\right)} \approx h^{\mu s},
\]
so that, in all cases,
\[
  \mathrm{I} \lesssim \max\{ h^{\mu s}, h^{1-s} \} \delta(z)^{-2s} .
\]

The rest of the proof follows exactly as that of Theorem~\ref{thm:ErrorLinear} (error estimate).
\end{proof}

With this proxy of the solution at hand we are ready to obtain error estimates.

\begin{thm}[error estimates]
\label{thm:ErrObstacle}
In the setting of Proposition~\ref{prop:RegObstacle} H\"older regularity), let $u$ solve \eqref{eq:TheObstacleHJ} and $u_h \in \Fespace^0$ solve \eqref{eq:TheObstacleh}. If $\Triang$ satisfies \eqref{eq:grading} and $\bvare$ is chosen as in Definition~\ref{defn:ChoiceBVareObstacle}, we have that
\[
\| u - u_h \|_\Linf \lesssim \frake(h),
\]
where the quantity $\frake(h)$ is defined in \eqref{eq:ProjError}.
\end{thm}
\begin{proof}
We will construct suitable super- and subsolutions to the discrete obstacle problem \eqref{eq:TheObstacleh} and apply Lemma~\ref{lem:DiscCompareObstacle} (discrete comparison) to conclude. Notice that, owing to \eqref{eq:ProjError}, there is a sufficiently large $C>0$ for which
\[
  u - C \frake(h) \leq R_h u \leq u + C \frake(h).
\]

\noindent \underline{Supersolution:} Let $u_h^+ = R_h u + C_1 \frake(h) \in \Fespace$, where the constant $C_1>0$ is to be chosen, and notice that $u_h^+ \geq 0$ in $\Omega^c$. Moreover, for $z \in \VertInt$, with the barrier function $b = \chi_\Omega$ from Lemma \ref{lem:barrier} (barrier), we get
\[
  \calL_{\eta,\bvare}[u_h^+](z) = \calL_{\eta,\bvare}[R_h u](z)  +  C_1 \frake(h) \calL_{\eta,\bvare}[b](z) >  \calL_\eta [u](z) \geq f(z).
\]
In addition, if $C_1 \geq C$, we have, for any $z \in\VertInt $,
\[
  u_h^+(z) \geq u(z) + (C_1 - C)\frake(h) \geq \psi(z) + (C_1 - C)\frake(h) \geq \psi(z).
\]
We have then shown that $u_h^+$ is a supersolution for the obstacle problem. This implies, via Lemma~\ref{lem:DiscCompareObstacle}, that for every $x \in \Omega$ we have
\[
  u_h(x) \leq u_h^+(x) \leq u(x) + (C+C_1)\frake(h),
\]
so that
\[
  u_h(x) - u(x) \lesssim \frake(h).
\]

\noindent \underline{Subsolution:} We now define $u_h^- = R_h u - C_2\frake(h) \in \Fespace$ where $C_2>0$ is to be chosen. Notice that $u_h^- \leq 0$ in $\Omega^c$. We will show that $u_h^-$ is a subsolution. To see this, let $z \in \VertInt$ and assume that $u_h^-(z) \geq \psi(z)$. Then,
\[
  \psi(z) \leq u_h^-(z) \leq u(z) - (C_2 - C)\frake(h) < u(z),
\]
provided $C_2 > C$. The fact that this inequality is strict shows that
\[
  \calL_{\eta,\bvare}[u_h^-](z) = \calL_{\eta,\bvare}[R_h u](z)  -  C_2 \frake(h) \calL_{\eta,\bvare}[b](z) <  \calL_\eta[u](z) = f(z),
\]
so that indeed $u_h^-$ is a subsolution. We invoke once again Lemma~\ref{lem:DiscCompareObstacle} to obtain
\[
  u(x) - (C + C_2)\frake(h) \leq u_h^-(x) \leq u_h(x), \quad \forall x\in \Omega,
\]
as we needed to show.
\end{proof}

\begin{remark}[optimality and complexity]
\label{rem:OptComplexObstacle}
We comment that the rate of convergence of Theorem~\ref{thm:ErrObstacle}, as expressed by the quantity $\frake(h)$ is optimal for our proof technique. To see this, we recall that near the boundary the solution to the obstacle problem behaves like that of the linear problem, \ie $u(z) \approx \delta(z)^s$, so that the rate of interpolation is at best $h^{\mu s}$. On the other hand, the interior regularity of the solution is, at best, $C^{1,s}$. Since our operator $\calL_\eta$ is of order $2s$, and our proof technique is based on comparison principles, the rate in the interior can be at best $h^{1+s-2s}= h^{1-s}$, as we have obtained.

Finally, with the notation and conventions of Remark~\ref{rem:ComplexityLinear} (complexity estimate) let us present the following estimates in terms of degrees of freedom
\[
  \| u - u_h \|_\Linf \lesssim 
  \begin{dcases}
    \max\left\{ N^{-s} |\log N|^s, N^{(s-1)/2} |\log N|^{(1-s)/2} \right\}, & d = 2, \\
    \max\left\{ N^{-s/2} |\log N|^{s/2}, N^{(s-1)/3} |\log N|^{(1-s)/3} \right\}, & d = 3.
  \end{dcases}
\]
\end{remark}

\subsection{Regularity of free boundaries}
\label{sub:FreeBoundaryObstacle}

Of particular importance in applications is the so-called free boundary, which is the boundary of the contact set 
\[
  \Gamma = \partial \Omega^0 \cap \Omega.
\]
In this section we are concerned with the regularity of this set. 

\subsubsection{Regular points}
We begin with the regularity of the free boundary near \emph{regular points} 
which, roughly speaking, are those at which the function $u-\psi$ grows at a 
rate $1+s$. Define $d(x) = \dist(x,\Omega^0)$ to be the distance from $x \in 
\Omega$ to the contact set $\Omega^0$. The following result is proved in \cite[Theorem 1.1]{Caffarelli2017}.

\begin{thm}[regular points I]
\label{thm:RegPoints}
Let $\alpha \in (0, \min\{s, 1-s\})$. Let $\psi \in \Psi$ and $x_0 \in \Gamma$ be a regular point. Then, there exist positive constants $a(x_0)$ and $r(x_0)$ such that
\[
  u(x)- \psi(x) = a(x_0) d(x)^{1+s} + \frako(|x-x_0|^{1+s+\alpha})
\]
for all $x \in B_{r(x_0)}(x_0) \cap \Omega^+$. Moreover, the set of points satisfying this property is an open set of $\Gamma$ and is locally a $C^{1,\gamma}$ graph for all $\gamma \in (0,s)$. Finally,
\[
  u \in C^{1,s}\left( B_{r(x_0)}(x_0) \right)
\]
\end{thm}

\begin{remark}[singular points]
The points of $\Gamma$ that are not regular are called \emph{singular}. According to \cite[Theorem 1.1]{Caffarelli2017}, these points satisfy
\[
  u(x) - \psi(x) = \frako( |x - x_0|^{1+s+\alpha} ).
\]
Singular points do in fact occur and are characterized in \cite{MR3783214} for the case $\eta \equiv 1$.
\end{remark}

\begin{remark}[$C^{1,\gamma}$ smoothness]
Consider a point $x_0 \in \Gamma$, the interface, with normal vector $\bnu$, and  $x =x_0 + r_0 \bnu$ for a sufficiently small $r_0 > 0$. If $\Gamma$ is $C^2$, then the closest point to $x$ in $\Gamma$ is $x_0$. However, if $\Gamma$ is of class $C^{1,\gamma}$ with $\gamma < 1$ (as in the conclusion of Theorem~\ref{thm:RegPoints} ), then this is not the case anymore and the distance $d(x,\Gamma)$ may be realized at a point different than $x_0$.
\end{remark}

The following variant of Theorem~\ref{thm:RegPoints}, which avoids the use of the distance function $d$, is stated in \cite[Theorem 4.4.1]{RosOtonBook}.

\begin{thm}[regular points II]
\label{thm:NewRegPoints}
Let $\alpha \in (0, \min\{s, 1-s\})$, $\theta > \max\{0,2s-1\}$, and $\gamma \in (0,s)$. Let the obstacle $\psi \in C_0^{2,\theta}(\R^d)$. Let $x_0 \in \Gamma$ be a free boundary point. Then, there is $r_0>0$  such that, for all $x \in B_{r_0}(x_0)$, we have:
\begin{enumerate}[(i)]
  \item either
  \begin{equation}
  \label{eq:RegPoint}
    u(x) -\psi(x) = a_0 \left( (x-x_0) \cdot \bnu \right)_+^{1+s} + \calO( |x-x_0|^{1+s+\gamma} ),
  \end{equation}
  for some $a_0>0$, $\bnu \in \polS^{d-1}$, and $\gamma > 0$,
  
  \item or
  \begin{equation}
  \label{eq:SingPoint}
    u(x) - \psi(x) = \calO( |x-x_0|^{1+s+\alpha}).
  \end{equation}
\end{enumerate}
Moreover, the set of points satisfying \eqref{eq:RegPoint} (regular points) is an open set of $\Gamma$, and it is locally a $C^{1,\gamma}$ manifold.
\end{thm}

The following result is presented in \cite[Corollary 4.5.3]{RosOtonBook} without proof. Since this estimate will be useful in our constructions we present a proof.

\begin{lemma}[H\"older continuity]
\label{lem:HolderContaNb}
Let $x_0 \in \Gamma$ be a regular point, namely, one that satisfies \eqref{eq:RegPoint}.
In the setting of Theorem~\ref{thm:NewRegPoints}, the vector and scalar fields
\[
  \bfb(z) = \lim_{x \to z} \frac1{d(x)^s} \GRAD (u(x) - \psi(x)), \qquad a(z) = \frac{|\bfb(z)|}{1+s}
\]
are of class $C^{0,\gamma}$ in $B_{r_0}(x_0) \cap \Gamma$.
\end{lemma}
\begin{proof}
Let $x_1,x_2 \in \Gamma \cap B_{r_0}(x_0)$ be two arbitrary free boundary points  which we assume to be regular points and so to satisfy Theorem~\ref{thm:NewRegPoints}. Since the set of regular points is open in $\Gamma$, this may require further restricting $r_0>0$.

Denote $v = u - \psi$. It is possible to show that \cite[Proposition 4.4.15]{RosOtonBook}
\[
  \left\| \frac1{d^s} \GRAD v\right\|_{C^{0,\gamma}(B_{r_0}(x_0))} \leq c_0,
\]
whence
\[
  \left| \frac1{d(x)^s} \GRAD v(x) - \frac1{d(y)^s} \GRAD v(y) \right| \leq c_0 |x-y|^\gamma, \qquad \forall x, y \in B_{r_0}(x_0).
\]
Since $\Gamma$ is $C^{1,\gamma}$ within $B_{r_0}(x_0)$ we let $\bnu_1,\bnu_2 \in \polS^{d-1}$ be the unit normals to $\Gamma$ at $x_1,~x_2$ pointing towards $\Omega^+$. Let $x \to x_1$, $y \to x_2$ and use the fact that
\[
  \lim_{x \to x_1} \frac1{d(x)^s} \GRAD v(x) = \bfb(x_1), \qquad \lim_{y \to x_2} \frac1{d(y)^s} \GRAD v(y) = \bfb(x_2),
\]
to deduce that
\[
  \left| \bfb(x_1) - \bfb(x_2) \right| \leq c_0 |x_1-x_2|^\gamma,
\]
whence $\bfb \in C^{0,\gamma}(B_{r_0}(x_0) \cap \Gamma)$. Since $a(x_i) =\tfrac1{1+s} |\bfb(x_i)|$, we infer that
\[
  |a(x_1) - a(x_2)| \lesssim | |\bfb(x_1)| - |\bfb(x_2)| | \leq |\bfb(x_1) - \bfb(x_2)| \leq c_0 |x_1 - x_2|^\gamma.
\]
This is the desired estimate for $a$.
\end{proof}

To exploit the previous result we make the following convenient, but realistic, regularity assumption on $\Gamma$.

\begin{assume}[regular points]
\label{assume:nonsing}
  The free boundary $\Gamma$ consists only of regular points, namely those that satisfy Theorem~\ref{thm:RegPoints} or \eqref{eq:RegPoint}.
\end{assume}

Next we discuss the fundamental \emph{nondegeneracy properties} (NDP). Given $\vare>0$ we let $\frakS(\Gamma,\vare)$ denote a strip of thickness $\vare$ around the free boundary $\Gamma$, namely,
\[
  \frakS(\Gamma,\vare) = \left\{ x \in \Omega \ \middle| \ d(x) = \dist(x,\Gamma) < \vare \right\}, \qquad 
  \frakS^+(\Gamma,\vare) = \frakS(\Gamma,\vare) \cap \Omega^+.
\]

The following result is known as an NDP in distance. It prescribes a pointwise behavior of $u - \psi$ (growth with rate at least $1+s$) if one is close to a regular point of the free boundary $\Gamma$.

\begin{col}[NDP in distance]
\label{cor:NDPDist}
Let $K \Subset \Omega$ and set $\tilde\Gamma = \Gamma \cap K$.
If Assumption~\ref{assume:nonsing} is valid, then there exists constants $a, \vare_0>0$ such that
\[
  u(x) - \psi(x) \geq a d(x)^{1+s}, \qquad \forall x\in \frakS^+(\tilde\Gamma,\vare_0).
\]
\end{col}
\begin{proof}
We proceed in several steps.
\begin{enumerate}[1.]
  \item Since, by Assumption~\ref{assume:nonsing}, every point $x_0 \in \tilde\Gamma$ is regular we have that $a(x_0) = \tfrac{|\bfb(x_0)|}{1+s} >0$, and $a$ and $\bfb$ are of class $C^{0,\gamma}$ in view of Lemma~\ref{lem:HolderContaNb}. We then deduce that
  \[
    a = \frac12 \min_{x_0 \in \tilde\Gamma} a(x_0) > 0,
  \]
  because $\tilde\Gamma$ is compact.
  
  \item Let $v = u-\psi$, $x_0 \in \tilde\Gamma$, and $r_0>0$ be such that
  \[
    \left\| \frac1{d^s} \GRAD v \right\|_{C^{0,\gamma}(B_{r_0}(x_0))} \leq C_0 = C(x_0).
  \]
  Given $x \in B_{r_0/2}(x_0)$ we let $x_1 \in \tilde\Gamma \cap B_{r_0}(x_0)$ be a point at minimal distance, \ie
  \[
    d(x) = |x-x_1| = (x-x_1) \cdot \bnu_1.
  \]
  Then, for $y$ close to $x_1$, we have
  \[
    \left| \frac1{d(x)^s} \GRAD v(x) - \frac1{d(y)^s} \GRAD v(y) \right| \leq C_0 |x-y|^\gamma,
  \]
  whence, upon computing the limit as $y \to x_1$, we get 
  \[
    \left| \frac1{d(x)^s} \GRAD v(x) - \bfb_1 \right| \leq C_0 |x-x_1|^\gamma,
  \]
  This implies that
  \[
    \frac1{d(x)^s} \partial_{\bnu_1} v(x) \geq \bfb_1 \cdot \bnu_1 - C_0 |x-x_1|^\gamma = |\bfb_1| - C_0 |x-x_1|^\gamma \geq (1+s) a,
  \]
  provided $r_0^\gamma C_0 \leq (1+s) a$, upon restricting $r_0$ if necessary. Therefore, we deduce the nondegeneracy property for the normal derivative
  \[
    \partial_{\bnu_1}v(x) \geq (1+s)a d(x)^s = (1+s) a \left( (x-x_1)\cdot \bnu_1 \right)_+^s.
  \]

  \item Let $x(t) = tx + (1-t) x_1$ denote any point in the segment joining $x_1$ and $x$, $t \in [0,1]$. Then $x_1 \in \tilde\Gamma$ is again a point in $\tilde\Gamma$ at a minimal distance to $x(t)$. The previous point implies then that
  \[
    \frac{\diff v(x(t))}{\diff t} = \partial_{\bnu_1} v(x(t)) |x-x_1| \geq (1+s)a t^s |x-x_1|^{1+s},
  \]
  or
  \[
    v(x(t)) = v(x(1)) - v(x(0)) = \int_0^1 \frac{\diff v(x(t))}{\diff t} \diff t \geq a|x-x_1|^{1+s}, 
  \]
  for all $x \in B_{r_0/2}(x_0)$. This is the desired local nondegeneracy property.
  
  \item We cover $\tilde\Gamma$ with balls $B_{r_0/2}(x_0)$ for every $x_0 \in \tilde\Gamma$. Since $\tilde\Gamma$ is compact, there is a finite subcovering
  \[
    \tilde\Gamma \subset \bigcup_{m=1}^M B_{r_m/2}(x_m).
  \]
  Finally, let $\vare_0>0$ be the distance from $\tilde\Gamma$ to the complement of $\cup_{m=1}^M B_{r_m/2}(x_m)$. Then every $x \in \frakS^+(\tilde\Gamma, \vare_0)$ belongs to a ball $B_{r_m/2}(x_m)$ for which the previous step applies.
\end{enumerate}
This concludes the proof.
\end{proof}

\begin{remark}[NDP]
Observe that Corollary~\ref{cor:NDPDist} implies the following weaker form of nondegeneracy: For all $x_0 \in \tilde\Gamma$ and $r \in (0,\vare_0]$ we have
\begin{equation}
\label{eq:wNDP}
  \sup_{x \in B_r(x_0)} \left( u(x) - \psi(x) \right) \geq a r^{1+s}.
\end{equation}
This inequality is also valid for all $\vare_0 < r \leq \diam(\Omega)$ because
\[
  \vare_0 = \frac{\vare_0}r r \geq \frac{\vare_0}{\diam(\Omega)} r,
\]
which yields \eqref{eq:wNDP} with the constant $a$ replaced by
\[
  \tildea = a \left( \frac{\vare_0}{\diam(\Omega)} \right)^{1+s}.
\]
\end{remark}

We now make an explicit assumption about the boundary behavior of $u-\psi$, which is not only useful for the subsequent argument, but it also has been used in deriving regularity via a localization argument; see Proposition~\ref{prop:RegObstacle} (H\"older regularity).

\begin{assume}[boundary behavior]
\label{assume:BdryBehavior}
The obstacle is strictly negative on $\partial\Omega$, \ie there is $c_0>0$ for which
\[
  \psi(x) \leq -c_0<0, \quad \forall x\in \partial\Omega.
\]
\end{assume}

The assumption above, by continuity, assumes that the free boundary $\Gamma$ is uniformly away from the boundary of the domain $\partial\Omega$, and thus the problem is linear in a neighborhood of $\partial\Omega$; see \cite{Borthagaray2019} for details in the case $\eta \equiv 1$. In addition, this implies that the NDP in distance of Corollary~\ref{cor:NDPDist} holds for all $x \in \frakS^+(\Gamma,\vare_0)$.

Let now $\vare \in (0,\vare_0]$. We define a $\vare$--neighborhood of the free boundary $\Gamma$
\[
  \frakN(\Gamma,\vare) = \left\{ x \in \Omega^+ \ \middle| \ 0 < u(x) - \psi(x) < \vare^{1+s} \right\}.
\]

\begin{col}[comparing $\frakN(\Gamma,\vare)$ and $\frakS(\Gamma,\vare)$]
\label{cor:NeighbVSStripNDP}
If Assumptions~\ref{assume:nonsing} and \ref{assume:BdryBehavior} hold, then there is $\vare_1 \in (0,\vare_0]$, with $\vare_0>0$ defined in Corollary~\ref{cor:NDPDist}, such that
\[
  \frakN\left(\Gamma, a^{\tfrac1{1+s}}\vare \right) \subset \frakS(\Gamma,\vare), \qquad \forall \vare \in (0,\vare_1].
\]
\end{col}
\begin{proof}
Consider the compact set $\omega$ between $\frakS^+(\Gamma,\vare_0)$ and $\partial\Omega$, namely,
\[
  \omega = \overline{ \Omega \setminus \left( \frakS^+(\Gamma,\vare_0) \cup \Omega^0 \right) }.
\]
Set again $v = u-\psi$. In view of Assumption~\ref{assume:BdryBehavior} and the fact that $v \in C^{0,s}(\bar\Omega)$, we deduce
\[
  v \geq c_0, \text{ on } \partial\Omega, \qquad v \geq a\vare_0^{1+s} \text{ on } \partial\frakS^+(\Gamma,\vare_0) \cap \Omega^+,
\]
and that there is a constant $c_1>0$ such that
\[
  v(x) \geq c_1 , \qquad \forall x \in \omega.
\]

Let $\vare_1 \in (0,\vare_0]$ be given by $a\vare_1^{1+s} = c_1$. The 
definition of $\frakS(\Gamma,\vare)$ in conjunction with  
Corollary~\ref{cor:NDPDist} (NDP in distance) guarantee that for all $\vare \in (0,\vare_1]$
\[
  v(x) \geq a \vare^{1+s}, \qquad \forall x\in \frakS(\Gamma,\vare)^c \cap \Omega^+.
\]
Therefore, for all $x \in \frakN(\Gamma,a^{\tfrac1{1+s}}\vare)$ we have
\[
  0 < v(x) < \left( a^{\tfrac1{1+s}} \vare \right)^{1+s} = a \vare^{1+s},
\]
whence $x \in \frakS(\Gamma,\vare)$ as asserted. This concludes the proof.
\end{proof}

We comment that the inclusion asserted in the previous Corollary is the typical assumption that is made for the error estimates in distance for the classical obstacle problem. See, for instance, \cite[Section 2(d)]{MR3393323}.

\subsubsection{Singular points}

We now examine nondegeneracy for singular points. We base our discussion on \cite{MR3783214}, which requires that Theorem~\ref{thm:RegPoints} holds. We recall that a singular point $x_0 \in \Gamma$ corresponds to 
\[
  a(x_0) = \frac{|\bfb(x_0)|}{1+s}= 0
\]

The characterization of singular points given in \cite{MR3783214} hinges on the following structural assumption.

\begin{assume}[singular points]
\label{assume:SingPoints}
There is $c_0 >0$ and $\gamma >0$ such that $\psi \in C^{3,\gamma}(\Omega)$, $f+\Delta \psi \leq -c_0 < 0$ in $\{ x \in \Omega \ | \ \psi(x) > 0 \}$ and $\emptyset \neq \{ x \in \Omega \ | \ \psi(x) > 0 \} \Subset \Omega$.
\end{assume}

For the rest of the discussion regarding singular points we restrict our attention to the \emph{fractional Laplacian}, \ie $\eta \equiv 1$. The following result is proved in \cite[Lemma 3.2]{MR3783214}.

\begin{lemma}[general growth]
\label{lem:GenGrowth}
If Assumption~\ref{assume:SingPoints} holds, then there are constants $a,r_0>0$ such that, for all $x_0 \in \Gamma$
\[
  \sup_{x \in B_r(x_0)} \left( u(x) - \psi(x) \right) \geq a r^2, \qquad \forall r \in (0, r_0).
\]
\end{lemma}

Notice that this is similar to \eqref{eq:wNDP} but with exponent $2$ instead of $1+s$.

Consider now the following class of homogeneous polynomials of degree two:
\[
  \polP_2^+ = \left\{ p_2(x) = \frac12 x^\intercal \bfA x \ \middle| \ \bfA \in \R^{d \times d}\setminus\{\bfO\}, \ \bfA^\intercal = \bfA, \sigma(\bfA) \subset [0,\infty) \right\}.
\]
The following result, which can be found in \cite[Proposition 7.2]{MR3783214}, is crucial to characterize singular points but does not play an important role in our discussion.

\begin{prop}[growth at singular points]
\label{prop:GSingP}
If Assumption~\ref{assume:SingPoints} holds, then there is a modulus of continuity $\omega : \R^+ \to \R^+$ such that for any $x_0 \in \Gamma$ singular, we have
\[
  u(x) - \psi(x) = p_2^{x_0}(x-x_0) + \omega( |x-x_0| )|x-x_0|^2,
\]
for some $p_2^{x_0} \in \polP_2^+$.
\end{prop}

\begin{remark}[growth at singular points]
In the setting of Proposition~\ref{prop:GSingP} it is important to notice that, if $\dim \ker \bfA = k$, with $k \in \{0, \ldots, d-1\}$, then $u-\psi$ exhibits strict quadratic growth in the directions orthogonal to $\ker \bfA$. Since there is at least one such direction, we conclude that Proposition~\ref{prop:GSingP} implies Lemma~\ref{lem:GenGrowth}. 

The first example of this scenario is when $\dim\ker\bfA = 0$. This corresponds to an isolated contact point $x_0$, with the function $u-\psi$ growing quadratically in all directions emanating from $x_0$.

As a second example we consider the situations given in Figure~\ref{fig:SingPoints}. Although the geometry is quite distinct, in both cases we have $\dim \ker \bfA = 1$. We expect quadratic growth in the directions of the red arrows. The case on the right is more degenerate than the case on the left, and worse for the approximation of the free boundary $\Gamma$. This will be discussed further below.
\end{remark}

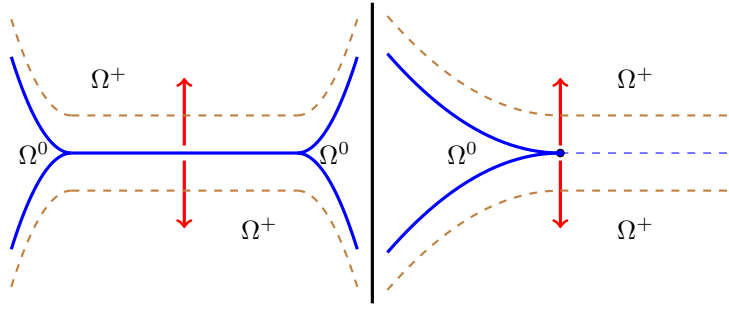
\begin{figure}[h!]
  \begin{center}
    \begin{tikzpicture}
      
      \draw[very thick, blue] (1,2) -- (4,2);
      \draw[very thick, blue, domain=0.2:1] plot (\x,{2*(\x -1)^2+2});
      \draw[very thick, blue, domain=0.2:1] plot (\x,{-2*(\x -1)^2+2});
      \draw[very thick, blue, domain=4:4.8] plot (\x,{2*(\x -4)^2+2});
      \draw[very thick, blue, domain=4:4.8] plot (\x,{-2*(\x -4)^2+2});
      
      \draw[very thick, red, ->] (2.5,2.1) -- (2.5,3);
      \draw[very thick, red, ->] (2.5,1.9) -- (2.5,1);
      
      \draw (0.5,2) node{$\Omega^0$};
      \draw (4.5,2) node{$\Omega^0$};
      
      \draw (1.5,3) node{$\Omega^+$};
      \draw (3.5,1) node{$\Omega^+$};
      
      \draw[thick, brown, dashed] (1,2.5) -- (4,2.5);
      \draw[thick, brown, dashed, domain=0.2:1] plot (\x,{2*(\x -1)^2+2.5});
      \draw[thick, brown, dashed, domain=4:4.8] plot (\x,{2*(\x -4)^2+2.5});
      \draw[thick, brown, dashed] (1,1.5) -- (4,1.5);
      \draw[thick, brown, dashed, domain=0.2:1] plot (\x,{-2*(\x -1)^2+1.5});
      \draw[thick, brown, dashed, domain=4:4.8] plot (\x,{-2*(\x -4)^2+1.5});

      \draw[very thick] (5,0) -- (5,4); 
      
      \draw[blue, thin, dashed] (7.5,2) -- (9.8,2);
      \draw[fill=blue,stroke=blue] (7.5,2) circle[radius = 0.05];
      \draw[very thick, blue, domain=5.2:7.5] plot (\x,{0.25*(\x-7.5)^2 + 2});
      \draw[very thick, blue, domain=5.2:7.5] plot (\x,{-0.25*(\x-7.5)^2 + 2});

      \draw[very thick, red, ->] (7.5,2.1) -- (7.5,3);
      \draw[very thick, red, ->] (7.5,1.9) -- (7.5,1);
      
      \draw (6.2,2) node{$\Omega^0$};
      
      \draw (8.5,3) node{$\Omega^+$};
      \draw (8.5,1) node{$\Omega^+$};
      
      \draw[brown, thick, dashed] (7.5,2.5) -- (9.8,2.5);
      \draw[brown, thick, dashed] (7.5,1.5) -- (9.8,1.5);
      \draw[thick, brown, dashed, domain=5.2:7.5] plot (\x,{0.25*(\x-7.5)^2 + 2.5});
      \draw[thick, brown, dashed, domain=5.2:7.5] plot (\x,{-0.25*(\x-7.5)^2 + 1.5});
      
    \end{tikzpicture}
  \end{center}
  \caption{Examples of singular free boundary points where the function $u-\psi$ has strict quadratic growth. We expect quadratic growth in the direction of the \textcolor{red}{red} arrows. The case on the right is more degenerate than the case on the left, and worse for the approximation of the free boundary $\Gamma$ (depicted in \textcolor{blue}{blue}). In both cases we have that $\dim \ker\bfA = 1$. A discrete free boundary $\Gamma_\Triang$, which is at a distance $\calO(\delta(h)^{1/2})$ is depicted in \textcolor{brown}{dashed brown}.}
  \label{fig:SingPoints}
\end{figure}

\subsection{Error estimates for free boundaries}
\label{sub:GammaEstimates}

Let us now put the pointwise error estimates of the previous sections to use and, on the basis of the discussions of Section~\ref{sub:FreeBoundaryObstacle}, obtain approximation properties for the free boundary $\Gamma$. To start we let $\psi_h = I_h \psi \in \Fespace$ be the Lagrange interpolant of $\psi$, namely $\psi_h(z) = \psi(z)$ for all $z \in \Vert$, and note that
\begin{equation}
\label{eq:IntErrorObstacle}
  \| \psi - \psi_h \|_\Linf = \sigma(h),
\end{equation}
for some function $\sigma : \R^+ \to \R^+$ with $\sigma(h) \downarrow 0$ as $h \downarrow 0$. Typically, and this shall be the case if $\psi \in \Psi$,
\[
  \sigma(h) \lesssim |\psi|_{W^{2,\infty}(\Omega)} h^2.
\]
Notice that \eqref{eq:TheObstacleh}, and its analysis, remain unchanged if we replace $\psi$ by $\psi_h$.

The next step is to define the discrete free boundary $\Gamma_h$. To do so, instead of looking at the zero level set of $u_h - \psi_h$ we consider the level set at height
\[
  \delta(h) = \tilde\frake(h) + \sigma(h),
\]
where $\tilde\frake(h) := \| u - u_h \|_\Linf$; see also \eqref{eq:ProjError}. Thus we define the discrete noncontact and contact sets, respectively, to be
\[
  \Omega_\Triang^+ = \left\{ x \in \Omega \ \middle| \ u_h(x) - \psi_h(x) > \delta(h) \right\}, \quad 
  \Omega_\Triang^0 = \left\{ x \in \Omega \ \middle| \ u_h(x) - \psi_h(x) \leq \delta(h) \right\}.
\]
The discrete free boundary is then,
\[
  \Gamma_\Triang = \partial \Omega^+_\Triang \cap \Omega = \partial\Omega_\Triang^0 \cap \Omega.
\]

We intend to prove that, in a sense, $\Gamma$ and $\Gamma_\Triang$ are close. This will be quantified by means of the \emph{Hausdorff distance}.

\begin{defn}[Hausdorff distance]
Let $A,B \subset \Real^d$. Their Hausdorff distance is
\[
  d_H(A,B) = \max \left\{ \max_{x \in A} \dist(x,B), \max_{y \in B} \dist(y,A) \right\}.
\]
\end{defn}

The error estimate for free boundaries reads as follows.

\begin{thm}[free boundary approximation]
\label{thm:ErrEstFreeBoundary}
Let $\Gamma$ satisfy Assumptions~\ref{assume:nonsing} and \ref{assume:BdryBehavior}, and let $a>0$ be given in Corollary~\ref{cor:NDPDist} and \eqref{eq:wNDP}. Then, there is $h_0>0$ such that
\[
  d_H(\Gamma, \Gamma_\Triang) \leq \left( \frac2a \delta(h) \right)^{\tfrac1{1+s}}, \qquad \forall h \in(0,h_0].
\]
\end{thm}
\begin{proof}
We proceed in three steps.

\begin{enumerate}[1.]
  \item We first show that $\Omega^0 \subset \Omega_\Triang^0$. To achieve this consider $x \in \Omega^0$, so that $u(x) = \psi(x)$. Thus, using \eqref{eq:ProjError} and \eqref{eq:IntErrorObstacle} we have
  \[
    u_h(x) - \psi_h(x) = \left( u_h(x) - u(x) \right) + \left( \psi(x) - \psi_h(x) \right) \leq \tilde\frake(h) + \sigma(h) = \delta(h).
  \]
  In other words, $x \in \Omega_\Triang^0$.
  
  \item Next we show that
  \[
    \Gamma_\Triang \subset \frakS\left(\Gamma, \left( \frac2a \delta(h) \right)^{\tfrac1{1+s}} \right).
  \]
  Let then $x \in \Gamma_\Triang$ and $x_0 \in \Gamma$ be the closest point to it, \ie
  \[
    |x-x_0| = d(x) = \dist(x,\Gamma).
  \]
  From the previous step we know that $x \in \Omega^+$. Since
  \[
    u_h(x) - \psi_h(x) = \delta(h), 
  \]
  and, in view of Theorem~\ref{thm:ErrObstacle} (error estimate), $u(x) - u_h(x) \leq \tilde\frake(h)$ we realize that
  \begin{align*}
    u(x) -\psi(x) &\leq \left( u(x) - u_h(x) \right) + \left(u_h(x) - \psi_h(x) \right) + \left( \psi_h(x) - \psi(x) \right) \\
      &\leq \tilde\frake(h) + \delta(h) + \sigma(h) = 2 \delta(h).
  \end{align*}
  Let now $h_0>0$ be sufficiently small so that $\delta(h_0) \leq \tfrac{a}2 \vare_1^{1+s}$, where $\vare_1>0$ is given in Corollary~\ref{cor:NeighbVSStripNDP} (comparing $\frakN(\Gamma,\vare)$ and $\frakS(\Gamma,\vare)$). This implies
  \[
    x \in \frakN(\Gamma, a^{\tfrac1{1+s}} \vare_1) \subset \frakS(\Gamma,\vare_1).
  \]
  Therefore, Corollary~\ref{cor:NDPDist} (NDP in distance) yields
  \[
    u(x) - \psi(x) \geq a \dist(x,\Gamma)^{1+s},
  \]
  whence,
  \[
    a \dist(x,\Gamma)^{1+s} \leq 2 \delta(h),
  \]
  \ie
  \[
    \dist(x,\Gamma) \leq \left( \frac2a \delta(h) \right)^{\tfrac1{1+s}},
  \]
  as desired.
  
  \item We show that
  \[
    \Gamma \subset \frakS\left(\Gamma_\Triang, \left( \frac2a \delta(h) \right)^{\tfrac1{1+s}} \right).
  \]
  Indeed, let $x_0 \in\Gamma$ and assume, for the sake of contradiction, that 
  \[
    R = \dist(x,\Gamma_\Triang) > \left( \frac2a \delta(h) \right)^{\tfrac1{1+s}}.
  \]
  Notice that the first step of the proof yields that if $B_R(x_0) \subset \Omega_\Triang^0$, then
  \[
    u_h(y) - \psi_h(y) \leq \delta(h), \qquad \forall y \in B_R(x_0).
  \]
  We now recall \eqref{eq:wNDP} which is a consequence of Corollary~\ref{cor:NDPDist}: for all $r \leq \diam(\Omega)$
  \[
    \sup_{y \in B_r(x_0)} \left( u(y) - \psi(y) \right) \geq a r^{1+s}.
  \]
  In other words, by compactness and continuity, there is $y \in B_R(x_0)$ such that 
  \[
    u(y) - \psi(y) \geq a R^{1+s}.
  \]
  
  On the other hand, 
  \begin{align*}
    u_h(y) - \psi_h(y) &= \left( u_h(y)-u(y) \right) + \left( u(y) - \psi(y) \right) + \left( \psi(y) - \psi_h(y) \right) \\
      &\geq -\tilde\frake(h) + aR^{1+s} - \sigma(h) > -\tilde\frake(h) + 2 \delta(h) - \sigma(h) > \delta(h),
  \end{align*}
  which is a contradiction. This proves the assertion and concludes the proof. \qedhere
\end{enumerate}
\end{proof}

\begin{remark}[stability]
We observe that the nondegeneracy constant $a>0$, defined in Corollary~\ref{cor:NDPDist}, acts as a stability parameter in the estimate of Theorem~\ref{thm:ErrEstFreeBoundary}.
\end{remark}

\begin{remark}[localized estimate]
\label{rem:LocEstFreeBdry}
Let $K \Subset \Omega$ be a compact set so that $\Gamma \cap K$ is made only of regular points. We thus allow $\Gamma$ to have singular points in $\Gamma \cap K^c$. Since the set of regular points is relatively open in $\Gamma$, we realize that Corollary~\ref{cor:NDPDist} is valid in $K$, \ie
\[
  u(x) -\psi(x) \geq a_K d(x)^{1+s}, \quad \forall x \in \frakS^+(\Gamma,\vare) \cap K,
\]
for some constant $a_K>0$ that, in particular, depends on the compact $K$. We can thus repeat the proof of Theorem~\ref{thm:ErrEstFreeBoundary} locally to deduce
\[
  d_H(\Gamma \cap K, \Gamma_\Triang \cap K ) \leq \left( \frac2{a_K} \delta(h) \right)^{\tfrac1{1+s}}, \qquad \forall h \in (0,h_0].
\]
\end{remark}

Theorem~\ref{thm:ErrEstFreeBoundary} and Remark~\ref{rem:LocEstFreeBdry} assume that, at least locally, there are no singular free boundary points. Let us conclude the discussion by presenting some results about the general case, namely when $\Gamma$ contains singular points. An inspection of the proof of Theorem~\ref{thm:ErrEstFreeBoundary} shows that the second step cannot be carried out, but the first and third one remain valid. In fact, the first step hinges on the definition of $\Gamma_\Triang$ and the last step relies on the growth condition \eqref{eq:wNDP} which, in principle, could be replaced by the quadratic growth condition provided in Lemma~\ref{lem:GenGrowth} (general growth). This brings about the following result.

\begin{thm}[error estimates for singular points]
If Assumption~\ref{assume:SingPoints} holds and $\eta \equiv 1$, then for every $x \in \Gamma$ we have that $x \in \Omega_\Triang^0$ and
\begin{equation}
\label{eq:diecinueve}
  \dist(x, \Gamma_\Triang) \leq \left( \frac2a \delta(h) \right)^{\tfrac1{2}}.
\end{equation}

\end{thm}

\begin{remark}[a posteriori error estimation]
Estimate \eqref{eq:diecinueve} establishes an error of $\Gamma$ relative to $\Gamma_\Triang$. This is the spirit of an \emph{a posteriori} error estimate. We refer to \cite{MR2139239} for similar estimates for the classical obstacle problem for the Laplace operator.
\end{remark}

\begin{remark}[regularity]
  Estimate \eqref{eq:diecinueve} requires no regularity of the free boundary $\Gamma$ beyond the nondegeneracy property of Lemma~\ref{lem:GenGrowth}, which relies on Assumption~\ref{assume:SingPoints}. The free boundary regularity stated in Proposition~\ref{prop:GSingP} (growth at singular points) is not needed to assert \eqref{eq:diecinueve}. It is then natural to wonder how the quadratic growth in certain directions established in Proposition~\ref{prop:GSingP} can be exploited for free boundary approximation.

Consider the approximations depicted in Figure~\ref{fig:SingPoints}.
In the left panel of Figure~\ref{fig:SingPoints} the discrete free boundary is uniformly close to $\Gamma$ and a global estimate of the form
\[
  \dist(x,\Gamma) \lesssim \delta(h)^{1/2}, \qquad \forall x \in \Gamma_\Triang
\]
is expected. On the other hand, the right panel of Figure~\ref{fig:SingPoints} shows that points of $\Gamma_\Triang$ may be far away from $\Gamma$ and, thus, the estimate above cannot be obtained.
\end{remark}

To conclude the discussion of approximation of free boundaries at singular points we mention that the scenarios depicted in Figure~\ref{fig:SingPoints} also illustrate the relevance of the localized error estimates alluded to in Remark~\ref{rem:LocEstFreeBdry}.

\section{A concave, fully nonlinear, nonlocal, problem}
\label{sec:HJB}

As a final application of our two-scale discretization, we shall consider a 
nonlocal Hamilton Jacobi Bellman equation. Let $s \in (0,1)$, and $\eta_i \in 
\calC(\lambda,\Lambda)$ ($i \in \{ 1, 2\}$) for some $\lambda, \Lambda$. Given $f \in 
C(\bar\Omega)$ we must find $u:\Real^d \to \Real$ such that
\begin{equation}
\label{eq:TheNLHJB}
  \min\left\{ \calL_{\eta_1}[u], \calL_{\eta_2}[u] \right\}= f, \quad \text{ in } \Omega,
  \qquad u = 0, \quad \text{ in } \Omega^c.
\end{equation}

Equations of this form have gathered a lot of attention in recent times. We refer the reader to, for instance, \cite{MR2494809,MR2831115,MR3542618} for details regarding the existence, uniqueness, and regularity of solutions. Regarding numerics, in the case $\Omega = \R^d$, see \cite{MR3940348,MR4250625,MR2525605,MR2405856,MR2391526,chowdhury2023precise,chowdhury2024discretization}. To our knowledge, however, no reference addresses the numerical approximation in the case of a bounded domain, as in \eqref{eq:TheNLHJB}. Our goal here will be to provide then the first convergent method for such problem. To achieve this we will get inspiration from \cite{MR0596541,MR0536953,Agnelli2018,Blanc2016}, and relate this problem to a sequence of obstacle problems.

\subsection{Existence and uniqueness}
\label{sub:ExistUniqueNLHJB}

Since it will be useful for our numerical purposes, we begin by discussing the existence of solutions. Below, by $\div$ we mean the remainder of integer division.

\begin{thm}[existence and uniqueness]
\label{thm:ExistenceHJB}
Assume that $f \in C(\bar\Omega)$ and that $\calL_{\eta_i}$ have a common supersolution, \ie there is $U \in C^2(\Real^d)$ for which
\[
  \calL_{\eta_i}[U] \geq f \quad \text{ in } \Omega, \qquad U \geq 0, \quad \text{ in } \Omega^c.
\]
Then, problem \eqref{eq:TheNLHJB} has a unique solution. Moreover, this solution can be obtained as the uniform limit of the following sequence: $u_0 \in \tHs$ solves
\[
  \calL_{\eta_1}[u_0] = f, \quad \text{ in } \Omega.
\]
For $k \in \polN$, let $i = k \div 2 +1$. The function $u_k \in \tHs$ solves
\begin{equation}
\label{eq:ContObstacleIteration}
  \min\left\{ \calL_{\eta_i}[u_k]- f, u_k - u_{k-1} \right\} = 0, \quad \text{ in }\Omega.
\end{equation}
\end{thm}
\begin{proof}
We split the proof in several steps.

\noindent \underline{Convergence:}
Notice that the sequence $\{u_k\}_{k \in \polN_0}$, as solutions of an obstacle problem with common right hand side, lie uniformly in $C^{0,s}(\bar\Omega)$. Moreover, by construction, the sequence is nondecreasing, \ie $u_k \geq u_{k-1}$ and bounded above. Indeed, since
\[
  \calL_{\eta_1}[U] \geq f = \calL_{\eta_1}[u_0]\quad \text{ in } \Omega, \qquad U - u_0 \geq 0, \quad \text{ in } \Omega^c,
\]
we must have, by comparison, that $U \geq u_0$. Assume next that, for some $k \in \polN$, we have $u_{k-1} \leq U$. Let $x \in \Omega$. We either have $u_k(x) = u_{k-1}(x) \leq U(x)$ or $u_k(x) > u_{k-1}(x)$. However, at such points we must have
\[
  \calL_{\eta_i}[u_k](x) = f(x) \leq \calL_{\eta_i}[U](x).
\]
Therefore, if we denote $\Omega_k^+ = \{ x \in \Omega \ | \ u_k(x) > u_{k-1}(x) \}$ we see that
\[
  \calL_{\eta_i}[U - u_k] \geq 0, \quad \text{ in } \Omega_k^+, \qquad U -u_k \geq 0, \quad \text{ in } (\Omega_k^+)^c.
\]
Consequently, by comparison, we must have that $u_k \leq U$ in $\Omega_k^+$. By Dini's theorem then, there is $u \in C(\bar\Omega)$ such that $u_k \rightrightarrows u$.

\noindent \underline{Solution:} The uniform convergence also implies that, for $i \in \{1,2\}$,
\[
  \calL_{\eta_i}[u] \geq f, \quad \text{ in } \Omega.
\]
On the other hand, if we show that for every $k \in \polN_0$
\begin{equation}
\label{eq:minLukleqf}
  \min\left\{ \calL_{\eta_1}[u_k], \calL_{\eta_2}[u_k] \right\} \leq f, \quad \text{ in } \Omega,
\end{equation}
we may pass to the limit and obtain that
\[
  \calL_{\eta_1}[u] \geq f, \quad \calL_{\eta_2}[u] \geq f, \quad \min\left\{ \calL_{\eta_1}[u], \calL_{\eta_2}[u] \right\} \leq f,
\]
and thus $u$ must be a solution.

\noindent \underline{Proof of \eqref{eq:minLukleqf}:} We argue by induction. By the way the iterative scheme is initialized, \eqref{eq:minLukleqf} holds for $k=0$. Assume now that the inequality holds for some $k \in \polN$. Consider the noncoincidence set
\[
  x_n \in \Omega_k^+ = \left\{ x \in \Omega \ \middle| \  u_{k+1}(x) > u_k(x) \right\}.
\]
By the complementarity conditions we must have, for some $i \in \{1,2\}$, that
\[
  \calL_{\eta_i}[u_{k+1}](x_n)=f(x_n).
\]
If, on the other hand, we consider the coincidence set
\[
  x_c \in \Omega_k^0 = \left\{ x \in \Omega \ \middle| \  u_{k+1}(x) = u_k(x) \right\} \cap \Omega,
\]
we see that $u_{k+1}(x_c) = u_k(x_c)$ and $u_{k+1}(y) \geq u_k(y)$ for all $y \in \R^d$. By the inductive hypothesis, there is $i_0 \in \{1,2\}$ for which
\[
  \calL_{\eta_{i_0}}[u_k](x_c) \leq f(x_c),
\]
then
\begin{align*}
  \calL_{\eta_{i_0}}[u_{k+1}](x_c) &= \vp \int_{\Real^d} \left( u_{k+1}(x_c) - u_{k+1}(y) \right) \frac1{|x-y|^{d+2s}} \ani\left(\frac{x_c-y}{|x_c-y|}\right) \diff{y} \\
  &= \vp \int_{\Real^d} \left( u_{k}(x_c) - u_{k+1}(y) \right) \frac1{|x-y|^{d+2s}} \ani\left(\frac{x_c-y}{|x_c-y|}\right) \diff{y} \\
  &\leq \vp \int_{\Real^d} \left( u_{k}(x_c) - u_{k}(y) \right) \frac1{|x-y|^{d+2s}} \ani\left(\frac{x_c-y}{|x_c-y|}\right) \diff{y} \\
  &= \calL_{\eta_{i_0}}[u_{k}](x_c) \leq f(x_c),
\end{align*}
as claimed.

\noindent \underline{Uniqueness:} Follows by comparison.
\end{proof}

\begin{remark}[obstacle]
Notice that the iterative scheme presented in Theorem~\ref{thm:ExistenceHJB} requires, at every step the solution of an obstacle problem like the one described in Section~\ref{sec:Obstacle}.
\end{remark}

\begin{remark}[supersolution]
Theorem~\ref{thm:ExistenceHJB} relies on the existence of a common supersolution for the operators. A possible common supersolution is
\[
  U(x) = \frac{A}2 |x|^2 + B.
\]
The constant $B$ can be chosen so that $U \geq 0$ in $\Omega^c$, whereas we can choose $A$, depending only on $d$, $s$, $\lambda$, $\Lambda$, and $-\| f \|_\Linf$, to obtain a supersolution.
\end{remark}

\begin{remark}[rate of convergence]
It is not known to us whether a rate of convergence for the iteration of Theorem~\ref{thm:ExistenceHJB} can be established.
\end{remark}

\subsection{A convergent scheme}
\label{sub:DiscreteHJB}

As a final application of our constructions we present a scheme for problem \eqref{eq:TheNLHJB}. Inspired by the proof of Theorem~\ref{thm:ExistenceHJB} we consider the following iterative scheme: $u_{h,0} \in \Fespace^0$ is such that
\[
  \calL_{\eta_1,\bvare}[u_{h,0}](z) = f(z), \quad \forall z \in \VertInt.
\]
For $k \in \polN$, let $i = k \div 2 + 1$. The function $u_{h,k} \in \Fespace^0$ solves
\begin{equation}
\label{eq:ObstacleIteration}
  \min\left\{ \calL_{\eta_i,\bvare}[u_{h,k}](z) -f(z), u_{h,k}(z) - u_{h,k-1}(z) \right\} = 0, \quad \forall z \in \VertInt.
\end{equation}

We immediately observe that, for every $k$, \eqref{eq:ObstacleIteration} has a unique solution. We would now like to obtain convergence of $u_{h,k}$ to $u$, the solution to \eqref{eq:TheNLHJB}. In order to achieve this, we introduce, for $k \in \polN$, the function $\tildeu_{h,k} \in \Fespace^0$ that is the solution of
\begin{equation}
\label{eq:ObstacleIterationTilde}
  \min\left\{ \calL_{\eta_i,\bvare}[\tildeu_{h,k}](z) -f(z), \tildeu_{h,k}(z) - I_h u_{k-1}(z) \right\} = 0, \quad \forall z \in \VertInt.
\end{equation}
Notice that $\tildeu_{h,k}$ is nothing but an approximation to the function $u_k$ from the scheme of Theorem~\ref{thm:ExistenceHJB}. As such we expect that $\tildeu_{h,k} \to u_k$ with a given rate. We quantify this by introducing the following assumption.

\begin{assume}[approximation]
\label{ASSume:BailForHJB}
There is a continuous function $\sigma : \Real^+ \to \Real^+$ such that $\sigma(h) \downarrow 0$ as $h \downarrow 0$ for which
\[
  \sup_{k \in\polN_0} \| u_k - \tildeu_{h,k} \|_\Linf \leq \sigma(h),
\]
where $\{u_k\}_{k \in \polN_0}$ are defined in Theorem~\ref{thm:ExistenceHJB} and $\{\tildeu_{h,k}\}_{k \in \polN_0}$ in \eqref{eq:ObstacleIterationTilde}.
\end{assume}

\begin{remark}[smoothness and compatibility]
While we would like to assert that the rate $\sigma(h)$ of Assumption~\ref{ASSume:BailForHJB} is that given by Theorem~\ref{thm:ErrObstacle} we are unable to prove this. The reason is that the error estimates of this Theorem hinge on Proposition~\ref{prop:RegObstacle} (H\"older regularity) which needs the obstacle to belong to the class $\Psi$. This means that, for every $k \in \polN$, we must be able to assert that:
\begin{enumerate}[1.]
  \item $u_k \in C^{2,1}(\Omega)$. While we are not able to verify this directly, we comment that this smoothness assumption is taken from \cite{Borthagaray2019}, which in turn follows the arguments of \cite{Caffarelli2008a,MR2270163}. This may not be sharp.
  
  \item $u_{k|\partial\Omega} <0$. However, we have $u_{k|\partial\Omega} = 0$.
\end{enumerate}
\end{remark}

Clearly
\[
  \tilde\frake(h) \leq \sigma(h),
\]
where $\tilde{\frake}(h) = \| u-u_h \|_{L^\infty(\Omega)}$; see also \eqref{eq:ProjError}.
Let us now show convergence.

\begin{thm}[convergence]
In the setting of Theorem~\ref{thm:ExistenceHJB} and under Assumption~\ref{ASSume:BailForHJB}, let $u$ solve \eqref{eq:TheNLHJB}, and $\{u_{h,k}\}_{k \in \polN_0, h>0}$ be the solutions to \eqref{eq:ObstacleIteration}. Assume that, as $k \uparrow \infty$ and $h \downarrow 0$, we have that
\[
  k \, \sigma(h) \to 0.
\]
Then,  $u_{h,k} \rightrightarrows u$.
\end{thm}
\begin{proof}
As mentioned above, the family $\{\tildeu_{h,k}\}_{k \in \polN_0} \subset \Fespace^0$, defined in \eqref{eq:ObstacleIterationTilde}, is an approximation of the obstacle problem defined in \eqref{eq:ContObstacleIteration}.

Consider now the difference $\tildeu_{h,k} - u_{h,k}$. We claim that
\begin{equation}
\label{eq:InterpolantHJB}
  \| \tildeu_{h,k} - u_{h,k} \|_\Linf \leq \| u_{h,k-1} - I_h u_{k-1} \|_\Linf.
\end{equation}
Indeed, if we define $w_h = \tildeu_{h,k} +\| u_{h,k-1} - I_h u_{k-1} \|_\Linf$ we see that
\[
  w_h \geq \tildeu_{h,k} = 0, \quad \text{ in } \Omega^c.
\]
Moreover, for $z \in \VertInt$,
\[
  w_h(z) \geq \tildeu_{h,k}(z) + u_{h,k-1}(z) - I_h u_{k-1}(z) \geq u_{h,k-1}(z)
\]
and $\calL_{\eta_i,\bvare}[w_h](z) \geq f(z)$. In other words, the function $w_h$ is a supersolution to the obstacle problem \eqref{eq:ObstacleIteration}. Lemma~\ref{lem:DiscCompareObstacle} then implies that $u_{h,k} \leq w_h$, \ie
\[
  u_{h,k} - \tildeu_{h,k} \leq \| u_{h,k-1} - I_h u_{k-1} \|_\Linf.
\]
A similar argument shows the lower bound.

Using Assumption~\ref{ASSume:BailForHJB} we now iterate \eqref{eq:InterpolantHJB} to obtain
\begin{align*}
  \| u_k - u_{h,k} \|_\Linf &\leq \| u_k - \tildeu_{h,k} \|_\Linf + \| \tildeu_{h,k} - u_{h,k} \|_\Linf \\
  &\leq \sigma(h) + \| u_{h,k-1} - I_h u_{k-1} \|_\Linf \\
  &\leq \sigma(h) + \| u_{k-1} - I_h u_{k-1} \|_\Linf + \| u_{h,k-1} - u_{k-1} \|_\Linf \\
  &\leq 2\sigma(h) + \| u_{h,k-1} - u_{k-1} \|_\Linf \\
  &\leq (k+1)\sigma(h) + \| u_{h,0} - u_0 \|_\Linf \leq (k+2)\sigma(h),
\end{align*}
where, in the last step, we used that $u_{h,0}$ is an approximation to the solution of the linear problem.

The triangle inequality, and the assumption on $k$ and $h$, yield the result.
\end{proof}

\appendix

\section{Approximation of integrodifferential operators of order $2s$ by second order differential operators}
\label{sec:Regularize}

To gain some intuition about what properties of the operator we would like to preserve after regularization, and to highlight the differences between the case of constant $\eta$ (\ie a fractional Laplacian) and a variable one, here we inspect the approximation of the integrodifferential operator
\begin{equation}\label{eq:Ieps}
  \calI_\vare[w](x) = \int_{B_\vare} \frac{ w(x+y) - 2w(x) + w(x-y) }{|y|^{d+2s}} \eta\left( \frac{y}{|y|} \right) \diff y
\end{equation}
by considering its action on a quadratic, \ie $ w \in \polP_2$. This may help in justifying our choice, showing the stark difference between our approach and that of \cite{Han2021,Han2023}, as well as providing some intuition into existing works that operate in a reverse way. For instance, \cite{MR2667633,Nochetto2018,MR3939307} approximate (local) second order elliptic differential operators by integral operators like $\calI_\vare$. We emphasize, however, that to obtain consistent, monotone, finite-difference schemes for arbitrary (local) second order elliptic operators, one requires the use of wide stencils \cite{Neilan2017}.

We begin by observing that, for $w \in \polP_2$,
\[
  w(x+y) - 2w(x) + w(x-y) = D^2 w(x) : y \otimes y,
\]
where $D^2w(x)$ is the Hessian of $w$ at $x$, $\otimes$ is the outer product, and $:$ is the Frobenius inner product. This means that
\begin{align*}
  \calI_\vare[w](x) &= \int_{B_\vare} D^2w(x): y \otimes y\frac1{|y|^{d+2s}} \eta\left( \frac{y}{|y|} \right) \diff y \\
    &= D^2w(x):\int_{B_\vare}  y \otimes y\frac1{|y|^{d+2s}} \eta\left( \frac{y}{|y|} \right) \diff y
    = \bfA:D^2w(x),
\end{align*}
where
\[
  \bfA = \int_{B_\vare}  y \otimes y\frac1{|y|^{d+2s}} \eta\left( \frac{y}{|y|} \right) \diff y \in \R^{d\times d}, \qquad \bfA^\intercal = \bfA.
\]

Consider now the particular case of $\eta \equiv 1$. The change of variables $y = \vare z$ with $z \in B_1$ shows that
\[
  \bfA = \vare^{2(1-s)} \int_{B_1} z \otimes z \frac1{|z|^{d+2s}} \diff z = \frac{\vare^{2(1-s)} \omega_d}{2d(1-s)} \bfI,
\]
where $\bfI$ is the identity matrix, and $\omega_d = |\polS^{d-1}|$. Therefore,
\[
  \calI_\vare[w](x) = \frac{\vare^{2(1-s)} \omega_d}{2d(1-s)} \LAP w(x), \qquad \forall w \in \polP_2.
\]
This approximation of $\mathcal{L}_1$ in the ball $B_\vare$, followed by centered finite differences,
  is the two-scale strategy advocated in \cite{Han2021,Han2023}. We next explore this approach for $\eta\ne1$
  and uncover some hidden difficulties.

Assume now that the coefficient $\eta$ is not constant, but even, \ie $\eta(z) = \eta(-z)$. Notice that this assumption is included in the class $\calC(\lambda,\Lambda)$ of Definition~\ref{def:ClassCoeff}. The change of variables $y = \vare z$ implies that
\[
  \bfA = \vare^{2(1-s)} \int_{B_1} z \otimes z \frac1{|z|^{d+2s}} \eta\left( \frac{z}{|z|} \right) \diff z.
\]
We see that $\bfA$ is symmetric positive definite. Let us use polar coordinates to obtain
\begin{align*}
  \bfA &= \vare^{2(1-s)}\int_{\polS^{d-1}} \theta \otimes \theta \eta(\theta) \int_0^1 \frac{r^2}{r^{d+2s}} r^{d-1} \diff r \diff \theta \\
  &= \frac{\vare^{2(1-s)}}{2(1-s)} \int_{\polS^{d-1}} \theta \otimes \theta \eta(\theta) \diff \theta = \frac{\vare^{2(1-s)}}{2(1-s)} \bfA_0,
\end{align*}
where, since $\eta(\theta)=\eta(-\theta)$, $\bfA_0$ is symmetric, but not necessarily diagonal.
In this case
\begin{equation}
\label{eq:DivFormNL}
  \calI_\vare[w](x) = \frac{\vare^{2(1-s)}}{2(1-s)} \bfA_0 : D^2 w(x) = \frac{\vare^{2(1-s)}}{2(1-s)} \DIV( \bfA_0\GRAD w(x) ), \qquad \forall w \in \polP_2.
\end{equation}

We quickly comment on the case $\eta(\theta) \neq \eta(-\theta)$, although this does not fit in the class $\calC(\lambda,\Lambda)$ of Definition~\ref{def:ClassCoeff}. Decomposing $\eta$ as the sum of its symmetric and skew-symmetric components, one readily obtains that the latter gives rise to a null integral. Therefore, such a case reduces to \eqref{eq:DivFormNL}, with $\bfA_0$ corresponding to the symmetric part of $\eta$.

Finally, we consider the general case of a coefficient of the form $\eta(x,\tfrac{z}{|z|})$. Then
\[
  \bfA_0 = \bfA_0(x) = \int_{\polS^{d-1}} \theta \otimes \theta \eta(x,\theta) \diff \theta,
\]
and
\begin{equation}
\label{eq:NonDivFormNL}
  \calI_\vare[w](x) = \frac{\vare^{2(1-s)}}{2(1-s)} \bfA_0(x) : D^2 w(x), \qquad \forall w \in \polP_2.
\end{equation}

\begin{remark}[anisotropy]
As we have already mentioned, it is not easy to construct approximations of second order anisotropic operators, like those in \eqref{eq:DivFormNL} and \eqref{eq:NonDivFormNL}, that are monotone. In particular, \eqref{eq:NonDivFormNL} requires wide stencils.
\end{remark}

Let us mention now what are the implications of the previous considerations for the approximation of the operator $\calL_\eta$, introduced in \eqref{eq:TheOperator}. First, owing to the symmetry of the coefficient $\eta$, which is part of Definition~\ref{def:ClassCoeff}, we have \eqref{eq:etaISsymmetric}, with $\calK(r) = r^{-d-2s}$. Next, if we assume that $w \in \polP_2$, we have
\begin{align*}
  \calL_{\eta,\vare}[w](x) = & - \frac12 \calI_\vare[w](x)
  \\ & + \frac12 \int_{\R^d\setminus B_\vare} \left( 2w(x) - w(x+y) - w(x-y) \right) \eta\left( \frac{y}{|y|} \right) \calK(|y|) \diff y,
\end{align*}
where $\calI_\vare[w]$ was defined in \eqref{eq:Ieps}.
Using \eqref{eq:DivFormNL} or \eqref{eq:NonDivFormNL} we then conclude that, if $w \in \polP_2$,
\begin{align*}
  \calL_{\eta,\vare [w](x)} = & - \frac{\vare^{2(1-s)}}{4(1-s)} \bfA_0 : D^2 w(x)
  \\
  & + \frac12 \int_{\R^d\setminus B_\vare} \left( 2w(x) - w(x+y) - w(x-y) \right) \eta\left( \frac{y}{|y|} \right) \calK(|y|) \diff y. 
\end{align*}

In summary, for quadratic functions our regularization \eqref{eq:DefOfKvare} is equivalent to replacing the singular part of the kernel $\frac{1}{|z|^{d+2s}} \eta(\frac{z}{|z|})$  in the ball $B_\vare$ by a scaled second order operator with symmetric {\it anisotropic} positive definite diffusion matrix $\bfA_0$. The special case $\eta = 1$ is the only one for which $\bfA_0=\bfI$ and this operator coincides with the scaled Laplace operator; moreover, centered finite differences
  are monotone and exact on quadratics. Otherwise, for $\eta\ne1$ a finite difference discretization requires wide stencils to be monotone \cite[Section 3.2]{Neilan2017}.  

\bibliographystyle{plain}
\bibliography{biblio}

\end{document}